\documentclass[11pt, reqno]{amsart}
\usepackage{amsmath}
\usepackage{amssymb}
\usepackage{esint, enumitem}
\usepackage[hidelinks]{hyperref}
\usepackage{color}
\newtheorem{theorem}{Theorem}[section]
\newtheorem{lemma}[theorem]{Lemma}
\newtheorem{proposition}[theorem]{Proposition}
\newtheorem{corollary}[theorem]{Corollary}
\newtheorem{problem}[theorem]{Problem}

\theoremstyle{definition}
\newtheorem{definition}[theorem]{Definition}

\theoremstyle{remark}
\newtheorem{remark}[theorem]{Remark}
\newtheorem{claim}{Claim}

\numberwithin{equation}{section}

\newcommand{\bb}[1]{\mathbb{#1}}
\newcommand{\innerproduct}[1]{\langle #1 \rangle}

\allowdisplaybreaks[4]

\title[Interior Estimates and Regularity for the scalar curvature equation]{Interior Estimates and Regularity for the Scalar Curvature Equation in Dimension 4}

\date{\today}

\author{Zhenyu Fan}
\address{School of Mathematical Sciences, Peking University, Beijing, 100871, P. R. China}
\email{fanzhenyu@stu.pku.edu.cn}

\author{Ravi Shankar}
\address{Department of Mathematics, Fine Hall, Princeton University, Princeton,
NJ}
\email{rs1838@princeton.edu}

\begin{document}


\begin{abstract}
   We establish a priori interior curvature estimates for the scalar curvature equation in dimension 4. Our approach relies on the doubling method introduced by Shankar and Yuan \cite{Shankar-Yuan-annals}. Key to the proof are an a priori doubling inequality under a small gradient assumption and the Alexandrov regularity for viscosity solutions; the latter is obtained via a new iterative rotation argument. Furthermore, we establish interior regularity for $C^1$ viscosity solutions.
\end{abstract}

\maketitle

\tableofcontents

\section{Introduction}

In this article, we study interior curvature estimates and regularity for a hypersurface $\Sigma^n \subset \mathbb{R}^{n+1}$ with positive scalar curvature in dimension $n=4$. By the Gauss equation, the intrinsic scalar curvature of $\Sigma$ is given by
\[  R=\sum_{1\leq i<j\leq n} \kappa_i \kappa_j = \sigma_{2}(\kappa),  \]
where $\kappa= (\kappa_1, \dots, \kappa_n) $ are the principal curvatures of $\Sigma$. Consequently, when $\Sigma$ can be locally represented as a graph over $B_r\subset \bb{R}^n$, this is equivalent to studying the following scalar curvature equation, also known as the $\sigma_2$-curvature equation:
 \begin{equation}\label{eq: scalar curvature eq}
        \sigma_{2}(\kappa)=\sum_{1\leq i<j\leq n}\kappa_i\kappa_j=1.
\end{equation}

In dimension two, \eqref{eq: scalar curvature eq} reduces to a Monge-Amp\`ere equation $\det D^2u = (1+|Du|^2)^2$, whose geometric origins lie in the prescribed Gaussian curvature problem (the Minkowski problem) and the Weyl problem in isometric embedding, see \cite{Nirenberg, Pogorelov, Pogorelov-book}. 

Equation \eqref{eq: scalar curvature eq} is of independent interest in the theory of fully nonlinear elliptic equations. In the seminal series of works \cite{CNS3,CNS5}, Caffarelli, Nirenberg and Spruck introduced the $\sigma_k$-Hessian and $\sigma_k$-curvature equations:
\[ \sigma_k(D^2u)= \sum_{1\leq i_1<\cdots<i_k\leq n} \lambda_{i_1}\cdots\lambda_{i_k}=1 \quad \text{and} \quad \sigma_k(\kappa(u))= \sum_{1\leq i_1<\cdots<i_k\leq n} \kappa_{i_1}\cdots\kappa_{i_k}=1, \]
where $\lambda_i's$ are the eigenvalues of the Hessian $D^2u$, and $\kappa_i's$ are the principal curvatures of the graph $\Sigma=(x,u(x))$. For $k=n$, both correspond to the Monge-Amp\`ere equation. Interior $C^2$ estimates for the two-dimensional Monge-Amp\`ere equation were established by Heinz \cite{Heinz}. In higher dimensions $n\ge 3$, Pogorelov's $C^{1, 1-\frac{2}{n}}$ singular solutions \cite{Pogorelov-book} preclude establishing a priori interior $C^2$ estimates and $C^2$ regularity for the Monge-Amp\`ere equation. Later, Urbas \cite{Urbas-c-eg} extended Pogorelov's counterexamples to $\sigma_k$-Hessian and $\sigma_k$-curvature equations for $n\ge k\ge 3$. This leaves a longstanding open problem:
\begin{problem}\label{PROB}
    For $n\ge 3$, whether a priori interior $C^2$ estimates and regularity can be established for the $\sigma_2$-Hessian equation $\sigma_2(D^2u)=1$ and the $\sigma_2$-curvature equation $\sigma_2(\kappa)=1$?
\end{problem}

To date, this problem is completely resolved in dimension three. Interior estimates for the Hessian equation were established by Warren and Yuan \cite{Warren-Yuan-3D}. Qiu \cite{Qiu-Hessian-24}, Zhou \cite{ZhouXC} and Xu \cite{XuYF} further extended them to the case with general right-hand sides. The corresponding estimates for the curvature equation were obtained by Qiu \cite{Qiu-curvature}. 

In dimension four, interior estimates have only been established for the Hessian equation: first by Shankar and Yuan \cite{Shankar-Yuan-annals}, and later extended by Fan \cite{Fan-Hessian} to the case with general right-hand sides. In this article, we partially resolve the aforementioned open problem for the curvature equation in dimension four. Our first result is the following curvature estimate:

\begin{theorem}\label{THM: curvature est}
    Let $\Sigma=(x,u(x))$ be a smooth $2$-convex graph over $B_1\subset \bb{R}^4$ satisfying the scalar curvature equation \eqref{eq: scalar curvature eq} on $B_1$. Then we have an implicit curvature estimate
    \begin{equation}\label{eq: curvature est}
        |\kappa(0)|\leq C \left(4, \|u\|_{C^1(B_1)}, \omega_{Du} \right),
    \end{equation}
    where $\|u\|_{C^1(B_1)}=\|u\|_{L^{\infty}(B_1)}+\|Du\|_{L^{\infty}(B_1)}$ and $\omega_{Du}$ denotes the modulus of continuity of $Du$. 
\end{theorem}

\begin{remark}
    In higher dimensions $n\ge 5$, the curvature estimate \eqref{eq: curvature est} can be established by the same approach under the dynamic semi-convex condition:
    \[ \dfrac{\kappa_{\mathrm{min}}}{H} \ge -c(n) \quad \text{with}\quad  c(n) = \dfrac{\sqrt{3n^2+1}- n +1}{2n}, \]
    where $\kappa_{\mathrm{min}}$ denotes the minimum principal curvature of $\Sigma$. We need this condition only for the almost Jacobi inequality in Section \ref{Section: Almost Jacobi Inequality} to hold in higher dimensions \cite{Shankar-Yuan-annals}.
\end{remark}

\begin{remark}
     The curvature estimate \eqref{eq: curvature est} depends on the modulus of continuity of the gradient. Hence, we cannot use \eqref{eq: curvature est} to directly derive the regularity of $C^1$ viscosity solutions to \eqref{eq: scalar curvature eq}  via the standard smooth approximation argument, since smooth approximations of a $C^1$ viscosity solution may not admit a uniform modulus of continuity of the gradient.
\end{remark}

Our next result independently establishes interior regularity for $C^1$ viscosity solutions.
\begin{theorem}[Regularity]\label{THM: regularity for C^1 solutions}
    Let $\Sigma=(x,u(x))$ be a $C^1$ $2$-convex graph over $B_1\subset \bb{R}^4$, and let $u$ be a viscosity solution to \eqref{eq: scalar curvature eq} on $B_1$, then $u$ is smooth in $B_1$.
\end{theorem}

    A priori gradient estimates for the scalar curvature equation \eqref{eq: scalar curvature eq} were established by Korevaar \cite{Korevaar}. Thus, continuous viscosity solutions to \eqref{eq: scalar curvature eq} are known to be Lipschitz. Our result improves the regularity of viscosity solutions from $C^1$ to smooth, so there still remains a gap between Lipschitz and $C^1$ regularity.  There is a Lipschitz counterexample of Caffarelli-Yuan \cite{Caffarelli-Yuan} to the Monge-Amp\`ere / Gauss curvature type equation.

    We first review the development of  interior $C^2$ estimates for the $\sigma_2$-Hessian and $\sigma_2$-curvature equations. In dimension two, for the Monge-Amp\`ere equation, Heinz \cite{Heinz} first established $C^2$ estimates via isothermal coordinates. In recent years, Chen, Han and Ou \cite{Chen-Han-Ou} provided a maximum principle proof, and Liu \cite{LiuJK} gave an alternative proof via the partial Legendre transform. Furthermore, using Guan-Qiu's test function in \cite{Guan-Qiu}, one can obtain another pointwise proof. 

    In dimension three, the Hessian equation $\sigma_2(D^2u)=1$ corresponds to the special Lagrangian equation with critical phase $\arctan \lambda_1 + \arctan\lambda_2 + \arctan \lambda_3 = \frac{\pi}{2}$. In this case, the gradient graph $(x,Du(x))$ is a minimal (in fact volume minimizing) submanifold in $\bb{R}^3\times\bb{R}^3$. Using this minimal surface structure, Warren and Yuan \cite{Warren-Yuan-3D} employed the integral method of Bombieri, De Giorgi and Miranda \cite{B-DG-M} for gradient estimates of codimension 1 minimal graphs to obtain the $C^2$ estimate for $u$. For the curvature equation, Qiu \cite{Qiu-curvature} discovered a similar ``special Lagrangian" structure, in which the ``Lagrangian" graph $(X,\nu)$ has bounded mean curvature in $(\bb{R}^4\times \bb{R}^4, \mathrm{d}x^2 + \mathrm{d}y^2)$. He then also obtained the curvature estimate via the integral method. Later, utilizing this ``special Lagrangian" structure, Qiu and Zhou \cite{Qiu-Zhou} further studied gradient and curvature estimates for the special Lagrangian curvature equation $\sum \arctan\kappa_i = \Theta$.  

    In dimension four, Shankar and Yuan made a breakthrough by introducing a new doubling method to establish interior $C^2$ estimates
    for the Hessian equation. In this article, we apply this doubling method to the curvature equation case. This method is discussed in detail later.

    In higher dimensions $n\ge 5$, Problem \ref{PROB} remains open. However, partial results have been obtained under certain convexity constraints on the solutions. For example, Guan and Qiu \cite{Guan-Qiu}, as well as Chen, Jian and Zhou \cite{Chen-Jian-Zhou} established interior $C^2$ estimates for convex solutions to both the Hessian and curvature equations. Regarding the Hessian equation specifically,  Mooney \cite{Mooney} derived the strict $2$-convexity for convex viscosity solutions, and consequently proved regularity for convex viscosity solutions, see also \cite{Chen-Jian-Tu-Zhou}; McGonagle, Song and Yuan \cite{M-Song-Yuan} established interior $C^2$ estimates for almost convex solutions via a compactness argument, and Shankar and Yuan \cite{Shankar-Yuan-reg} later proved their regularity using the Legendre-Lewy transform; additionally, Shankar and Yuan \cite{Shankar-Yuan-semi-coonvex} also established interior $C^2$ estimates for semi-convex solutions via an integral method.

    Furthermore, for the Hessian equation, Urbas \cite{Urbas-Hessian-00, Urbas-Hessian-01}, Shankar-Yuan \cite{Shankar-Yuan-semi-coonvex} and Mooney \cite{Mooney-rmk-25} established $C^2$ estimates in terms of certain integrals of the Hessian. For the curvature equation, Urbas \cite{Urbas-curvature-00, Urbas-curvature-03} established similar curvature estimates in terms of certain integrals of the mean curvature; notably, his curvature estimates already exhibited the dependence on the modulus of continuity of the gradient.

    We now explain our strategy for proving Theorem \ref{THM: curvature est} and Theorem \ref{THM: regularity for C^1 solutions} using the doubling method. The doubling method consists of two parts: a doubling inequality and partial regularity. The first part is due to an observation by Qiu \cite{Qiu-Hessian-24}. For the three-dimensional $\sigma_2$-Hessian equation $\sigma_{2}(D^2u)=1$, Qiu proved that when the Jacobi inequality
    \begin{equation}\label{eq: Jacobi}
        F^{ij} \partial^{2}_{ij}(\log \sigma_1) \ge \varepsilon F^{ij}\partial_i(\log \sigma_1)\partial_i(\log \sigma_1)
    \end{equation}
    is valid, where $F^{ij}$ denotes the linearized operator, then the following doubling inequality 
    \begin{equation}
        \sup_{B_1} \sigma_1 \leq C\left(n, \|u\|_{C^1} \right) \sup_{B_{1/2}} \sigma_1
    \end{equation}
    can be established via the maximum principle using Guan-Qiu's test function involving $x\cdot Du -u$ \cite{Guan-Qiu}. By rescaling and iteration, this doubling inequality implies that controlling the Hessian on a small ball yields control at larger scales.  In dimension three, the Jacobi inequality \eqref{eq: Jacobi} holds for $\varepsilon=1/3$. In dimension four, we can only establish an ``almost Jacobi" inequality which means that $\varepsilon$ may degenerate to 0 somewhere. Nevertheless, Shankar and Yuan \cite{Shankar-Yuan-annals} still proved the doubling inequality via a slight modification of Qiu's argument. 
    
    In the context of our four-dimensional scalar curvature equation, an almost Jacobi inequality can still be established, albeit with an additional commutator term $-F^{ij}h_{ik}h_{jk}$, which can be controlled in the maximum principle argument. However, when applying the linearized operator of the curvature equation to Guan-Qiu's test function, the term $\Delta_{F}(x\cdot Du-u)$ introduces a bad term that is difficult to control. Note that the coefficient of this bad term depends on $|Du|$, hence we additionally assume that the gradient $|Du|$ is sufficiently small, making this bad term small enough to be handled. Consequently, we can establish a doubling inequality under a small gradient assumption, see Proposition \ref{PROP: doubling}.

    To obtain the control at small scales, we also need to establish partial regularity. For the Hessian equation, Shankar and Yuan \cite{Shankar-Yuan-annals} modified the proof of the Alexandrov theorem \cite{Evans-Gariepy, Chaudhuri-Trudinger} to prove the almost everywhere twice differentiability of viscosity solutions to $\sigma_{2}(D^2u)=1$. Combining this with Savin's small perturbation theorem ($\varepsilon$-regularity theorem) \cite{Savin} yields a partial regularity result: the singular set of viscosity solutions to $\sigma_2(D^2u)=1$ has Lebesgue measure zero. Finally, a compactness argument gives an implicit $C^2$ estimate.
    
    For the scalar curvature equation, establishing Alexandrov regularity for viscosity solutions is the most challenging part of this work. Recall from \cite{Evans-Gariepy} that the proof of the Alexandrov theorem mainly consists of two steps: interpreting the distributional Hessian as Radon measures, and establishing linearly dependent gradient estimates of the form:
    \begin{equation}\label{eq: 1.5}
        \sup_{B_{1/2}} |D(u-l)| \leq C\,  \underset{B_{1}}{\mathrm{osc}} (u-l), 
    \end{equation}
    where $l$ is the linear part of $u$. For the Hessian equation, interpreting the distributional Hessian as Radon measures is guaranteed by the $2$-convexity of admissible solutions, and the linearly dependent gradient estimates \eqref{eq: 1.5} were proved by Trudinger \cite{Trudinger-CPDE} and Chou-Wang \cite{Chou-Wang}. For the scalar curvature equation, the $2$-convexity of admissible solutions still ensures the Radon measure interpretation. However, the linearly dependent gradient estimates are difficult to obtain. The first difficulty is that the known gradient estimates for curvature equations established by Korevaar \cite{Korevaar} exhibit quadratic exponential dependence. To obtain linear dependence, we note that Korevaar's result provides an a priori gradient bound $\|Du\|_{L^{\infty}(B_1)}=\Gamma$. We can then improve this bound to a linearly dependent gradient estimate for $u$ at smaller scales:
    \begin{equation}\label{eq: 1.6}
        \sup_{B_{r/2}(x)} |Du| \leq \dfrac{C(n,\Gamma)}{r}\,  \underset{B_{r}(x)}{\mathrm{osc}} u, \quad \text{for any ball}\ B_r(x)\subset B_1.
    \end{equation}
    The constant $C$ here naturally depends on the a priori gradient bound $\Gamma$. The second difficulty is that after subtracting the linear part from $u$, the function $u-l$ no longer satisfies the scalar curvature equation \eqref{eq: scalar curvature eq}. Hence, \eqref{eq: 1.6} cannot provide the desired estimate for $u-l$. To overcome this  difficulty, our strategy is to rotate the graph $\Sigma=(x, u(x))$ so that $Du$ vanishes at a given point, thereby eliminating the need to subtract a linear function, Applying \eqref{eq: 1.6} to the rotated function then yields the desired estimate \eqref{eq: 1.5} for $u-l$. A subtle issue here is that $\Sigma$ is only locally graphical after rotation. By a rough estimate, the size of the domain where $\Sigma$ remains graphical after rotation depends on the modulus of continuity of $Du$. Therefore, we can obtain \eqref{eq: 1.5} with the constant $C$ depending on the a priori gradient bound $\Gamma$ and the modulus of continuity of the gradient, see Section \ref{subsec: Linearly dependent gradient estimates}. This suffices to derive the curvature estimate \eqref{eq: curvature est} in the final compactness argument.

    In Section \ref{Section: Alexandrov Regularity Revisited}, we refine our rotation argument. Through a finite number of iterative rotations, we obtain a more precise estimate for the size of the domain where $\Sigma$ remains graphical after rotation; in particular, this new estimate is independent of the modulus of continuity of the gradient (see Lemma \ref{LEMMA: rotation revisited}). This allows us to establish Alexandrov regularity for all viscosity solutions to the scalar curvature equation \eqref{eq: scalar curvature eq}. Consequently, we can directly establish interior regularity for $C^1$ viscosity solutions to \eqref{eq: scalar curvature eq} via the doubling argument. Although smooth approximations $\{u_k\}$ of a  $C^1$ viscosity solution $u$ may not necessarily share the curvature estimate \eqref{eq: curvature est}, combining Alexandrov regularity with Savin's small perturbation theorem \cite{Savin, Fan} yields a uniform curvature bound for $u_k$ at sufficiently small scales. Furthermore, by virtue of the $C^1$ assumption on $u$, we can use the linearly dependent gradient estimate \eqref{eq: 1.6} to verify that the small gradient assumption in Proposition \ref{PROP: doubling} holds uniformly for $u_k$ at some universal scale. Therefore, the doubling inequality provides a uniform curvature bound for $u_k$ at larger scales, thereby implying interior regularity for the $C^1$ viscosity solution $u$. 

    Finally, we remark that this doubling method can be employed to establish interior $C^2$ estimates and regularity for many other nonlinear PDEs. We refer to \cite{Shankar-Yuan-M-A} for the Monge-Amp\`ere equation; \cite{Jiao-Sui, CYFung-quotient} for the Hessian quotient equation; \cite{Shankar-sLag, CYFung-sLag} for the special Lagrangian equation; \cite{AB-RS-JW} for the Lagrangian mean curvature equation; and \cite{Shankar-PAMS} for general fully nonlinear PDEs.

\bigskip
\noindent
{\bf Organization.} This paper is organized as follows. In Section \ref{Section: Preliminaries}, we introduce some preliminaries related to the $\sigma_2$-operator and hypersurfaces geometry. In Section \ref{Section: Almost Jacobi Inequality}, we establish an almost Jacobi inequality, which is then used in Section \ref{Section: Doubling ineq} to derive a doubling inequality under the small gradient assumption. Section \ref{Section: Alexandrov Regularity} is devoted to establishing the Alexandrov-type theorem by deriving an a priori $W^{2,1}$ estimate and a linearly dependent gradient estimate, where the latter estimate initially involves the modulus of continuity of the gradient. In Section \ref{Section: Proof of Curvature Estimates}, we present the proof of Theorem \ref{THM: curvature est}. In Section \ref{Section: Alexandrov Regularity Revisited}, we refine the proof of the Alexandrov-type theorem to remove the dependence on the modulus of continuity of the gradient. Finally, the proof of Theorem \ref{THM: regularity for C^1 solutions} is provided in Section \ref{Section: Regularity}.

\bigskip
\noindent
{\bf Acknowledgments.}  
The authors thank Yu Yuan for conversations on this problem.

\section{Preliminaries}\label{Section: Preliminaries} 
In this section, we introduce the definitions, notation, and lemmas needed for the subsequent sections. 

We first introduce some algebraic lemmas related to the $\sigma_2$ operator, which can be found in \cite{Lin-Trudinger,Shankar-Yuan-annals, Fan-Hessian}.

\begin{lemma}\label{LEMMA: control on min eigenvalue}
    Let $\lambda=(\lambda_1,\dots,\lambda_n) \in \Gamma_2:= \{ \sigma_1(\lambda)>0, \sigma_2(\lambda)>0  \} $ with $\lambda_1\ge\lambda_2\ge\cdots\ge\lambda_n$. Then for $n>2$, the following sharp bound holds:
    \[  \sigma_1> \dfrac{n}{n-2}|\lambda_n|.  \]
\end{lemma}

\begin{lemma}\label{LEMMA: est on linearized operator}
    Let $\lambda=(\lambda_1,\dots,\lambda_n) \in \Gamma_2 $ with $\lambda_1\ge\lambda_2\ge\cdots\ge \lambda_n$. Then we have
    \begin{equation}
        \dfrac{\sigma_2}{\sigma_1}\leq \sigma_1-\lambda_1\leq \dfrac{n-1}{n}\sigma_1,
    \end{equation}
    and for any $i\ge 2$,
    \begin{equation}
        \left(1-\dfrac{1}{\sqrt{2}}\right)\sigma_1\leq \sigma_1-\lambda_i\leq \dfrac{2n-2}{n}\sigma_1.
    \end{equation}
\end{lemma}

Next, we introduce some notation from hypersurface geometry.
\begin{definition}
    A $C^2$ graph $\Sigma=(x,u(x))$ is called $2$-convex, or admissible, if at every point, its principal curvatures satisfy
    \begin{equation}\label{eq: 2.3}
        \kappa=(\kappa_1,\dots, \kappa_n)\in\Gamma_2.
    \end{equation}
    This definition can be extended to continuous graphs if $\eqref{eq: 2.3}$ holds in the viscosity sense.
\end{definition}

Suppose that a smooth 2-convex graph $\Sigma=(x,u(x))\subset \bb{R}^{n+1}$ satisfies the scalar curvature equation 
\[   \sigma_2(\kappa)=\sum_{i<j}\kappa_i\kappa_j=1.  \]
Let $X=(x,u(x))$ be the position vector on $\Sigma$, and  denote the outer (downward) unit normal of $\Sigma$ by $\nu=\frac{1}{W}(Du,-1) $, where $W=\sqrt{1+|Du|^2}$.

Let $\{E_{1},\dots, E_{n},E_{n+1}\}$ be the standard orthonormal coordinate of $\bb{R}^{n+1}$. Sometimes we may choose an orthonormal frame $\{e_1,e_2,\dots, e_n, \nu\}$ in $\bb{R}^{n+1}$ such that $\{e_1,e_2,\dots,e_n\}$ are tangent to $\Sigma$ and $\nu$ is the unit normal on $\Sigma$. We use $\nabla$ to denote the Levi-Civita connection on $\Sigma$. For any smooth function $f$ on $\Sigma$, we denote $f_i= \nabla_{e_i}f$ and $f_{ij}=\nabla^2f(e_i,e_j)$.  The following fundamental equations for hypersurfaces in $\bb{R}^{n+1}$ are well known:
\begin{align*}
    X_{i}&=e_i,\quad X_{ij}=-h_{ij}\nu \qquad \text{(Gauss formula)}\\
    \nu_i&=h_{ij}e_j \qquad \text{(Weingarten equation)}\\
    h_{ijk}&=h_{ikj}\qquad \text{(Codazzi equation)}\\
    R_{ijkl}&= h_{ik}h_{jl}-h_{il}h_{jk}\qquad \text{(Gauss equation)}
\end{align*}
where $h_{ij}$ is the second fundamental form of $\Sigma$, $h_{ijk}=\nabla_{e_k}h_{ij},$ and $R_{ijkl}$ is the curvature tensor of $\Sigma$. We also have the following commutator formula:
\[ h_{ijkl}-h_{ijlk}=\sum_{m}h_{im}R_{mjkl}+\sum_{m}h_{mj}R_{mikl}. \]
Combining this with the Gauss and Codazzi equations, we have
\begin{equation}\label{eq: commutator formula}
    h_{iikk}=h_{kkii}+\sum_{m}(h_{mi}^2h_{kk}-h_{mk}^2h_{ii}).
\end{equation}

In this orthonormal frame, the scalar curvature equation becomes 
\[  \sigma_{2}(\kappa)=\sigma_{2}(\lambda(h_{ij}))= \dfrac{1}{2}\left[ \left( \sum_{i}h_{ii} \right)^2-\sum_{i,j}h_{ij}^2 \right]=\dfrac{1}{2}\left( H^2 -|A|^2 \right)=1, \]
where $H$ and $|A|$ denote the mean curvature and the norm of the second fundamental form of $\Sigma$, respectively. Taking the covariant derivative with respect to $e_k$, we get the linearized equation:
\begin{equation}\label{eq: linearized eq}
    F^{ij}h_{ijk}=0,
\end{equation}
where
\begin{equation}
    F^{ij}:=\dfrac{\partial \sigma_2}{\partial h_{ij}}=H\delta_{ij}-h_{ij}.
\end{equation}
By Lemma \ref{LEMMA: est on linearized operator}, the coefficient matrix $(F^{ij})$ is positive definite, provided $\kappa=\lambda(h_{ij})\in \Gamma_2$, i.e., $\Sigma$ is 2-convex.

\section{Almost Jacobi inequality}\label{Section: Almost Jacobi Inequality}

In this section, we establish the almost Jacobi inequality for the quantity $b=\log\sigma_{1}=\log H$. We denote
\[ \Delta_{F}b:= F^{ij}b_{ij}\quad\text{and} \quad |\nabla_{F}b|^2:= F^{ij}b_i b_j. \]

\begin{proposition}
    Let $\Sigma=(x,u(x))$ be a smooth $2$-convex graph satisfying the scalar curvature equation \eqref{eq: scalar curvature eq} on $B_{1}\subset \bb{R}^4$, and let $b=\log H$, where $H=\sigma_1(\kappa)$ is the mean curvature of $\Sigma$. Then we have the following almost Jacobi inequality:
    \begin{equation}\label{eq: almost Jacobi}
        \Delta_{F}b\ge \varepsilon |\nabla_{F}b|^2 - \sum_{i,j,k}F^{ij}h_{ik}h_{jk},
    \end{equation}
    where
    \begin{equation*}
        \varepsilon= \dfrac{2}{9}\left( \dfrac{1}{2} + \dfrac{\kappa_{\mathrm{min}}}{H} \right)>0.
    \end{equation*}
    Here, $\kappa_{\mathrm{min}}$ is the minimum principal curvature of $\Sigma$.
\end{proposition}

\begin{proof}
    We prove the almost Jacobi inequality \eqref{eq: almost Jacobi} pointwise. Fix an arbitrary point $p$. By appropriately choosing the tangent frame $\{e_1,\dots,e_4\}$, we may assume that the second fundamental form $(h_{ij})$ is diagonal at $p$ with $h_{ii}=\kappa_i$ and $\kappa_1\ge \cdots\ge \kappa_4$. Under this frame, the linearized operator $(F^{ij})$ is also diagonal at $p$ with $F^{ii}=H-\kappa_i$. Hereafter, all computations are carried out at the point $p$. 

    Computing the derivatives of $b=\log H$, we obtain
    \[ |\nabla_{F}b|^2=\sum_i \dfrac{F^{ii}H_{i}^2}{H^2}\quad \text{and} \quad \Delta_{F}b=\sum_i \dfrac{F^{ii}H_{ii}}{H}-\sum_{i}\dfrac{F^{ii}H_{i}^2}{H^2}. \]
    Differentiating the linearized equation \eqref{eq: linearized eq} with respect to $e_k$ yields
    \begin{equation}
        \sum_{i}F^{ii}h_{iikk}- \sum_{i,j} h_{ijk}^2+ \left(\sum_{i} h_{iik}\right)^2=0
    \end{equation}
    Combining this with the commutator formula \eqref{eq: commutator formula}, we get
    \begin{align*}
        \sum_{i}F^{ii}H_{ii}&=\sum_{i,k}F^{ii}h_{kkii}=\sum_{i,k}F^{ii}(h_{iikk}+\kappa_k^2\kappa_i-\kappa_i^2\kappa_k)\\
        &= \sum_{i,j,k}h_{ijk}^2-\sum_k\left(\sum_i h_{iik} \right)^2 + \sum_{i,k} F^{ii}(\kappa_k^2\kappa_i - \kappa_i^2\kappa_k).
    \end{align*}

    For the last two commutator terms, using $\sum_i F^{ii}\kappa_i= 2$, we estimate
    \[\sum_{i,k} F^{ii}(\kappa_k^2\kappa_i - \kappa_i^2\kappa_k)\ge -H\sum_i F^{ii}\kappa_i^2. \]
    Now, combining the above estimates and setting $\delta := 1+\varepsilon$, we obtain
    \begin{equation}
        \begin{split}
            \Delta_{F}b-\varepsilon|\nabla_{F}b|^2 &\ge \dfrac{1}{H}\left\{    \sum_{i,j,k}h_{ijk}^2 - \sum_{i} \left( 1+\delta \dfrac{F^{ii}}{H} \right) H_{i}^2  \right\} - \sum_{i}F^{ii}\kappa_i^2 \\
            &\ge \dfrac{1}{H}\sum_{i} \left\{   3\sum_{j\neq i} h_{jji}^2 + h_{iii}^2 - \left( 1+\delta\dfrac{F^{ii}}{H} \right)H_{i}^2   \right\} - \sum_i F^{ii}\kappa_i^2\\
            &:= \dfrac{1}{H}\sum_{i}Q_i -\sum_i F^{ii}\kappa_i^2.
        \end{split}
    \end{equation}
    It remains to show that each $Q_i$ is nonnegative, which will complete the proof of the almost Jacobi inequality \eqref{eq: almost Jacobi}. Fix any $i\in\{ 1,2,3, 4 \}$. Let $t=t_i=(h_{11i}, \cdots, h_{44i})$, let $E_i$ be the standard $i$-th basis of $\bb{R}^4$, and set $a=(1,1,1,1)$. We view $Q_i$ as a quadratic form in $t$:
    \begin{equation*}
        Q_i = 3|t|^2 - 2\innerproduct{t,E_i}^2- \left(1+ \delta\dfrac{F^{ii}}{H}\right)\innerproduct{t,a}^2
    \end{equation*}
    The proof of the nonnegativity of $Q_i$ is identical to that in \cite[pp. 495-498]{Shankar-Yuan-annals}.

\end{proof}

\section{Doubling inequality under a small gradient assumption}\label{Section: Doubling ineq}

   In this section, we use the almost Jacobi inequality, combined with Guan-Qiu's test function in \cite{Guan-Qiu,Qiu-Hessian-24}, to establish an a priori doubling inequality. 
   
   Unlike Hessian equations, the curvature equation has a more complicated structure, which introduces an extra negative term in our maximum principle argument that does not appear in the Hessian equation case treated in \cite{Shankar-Yuan-annals}. Therefore, we need to impose an additional assumption that the gradient is sufficiently small to handle this bad term.  

   Furthermore, to facilitate our subsequent applications involving rescaling, we state the proposition for equations with an arbitrary constant right-hand side.

  \begin{proposition}\label{PROP: doubling}
      Let $\Sigma=(x,u(x))$ be a smooth $2$-convex graph over $B_1\subset \bb{R}^4$ satisfying the scalar curvature equation
      \begin{equation}
          \sigma_{2}(\kappa) = f_0 \quad \text{on}\quad B_3,
      \end{equation}
      for some positive constant $f_0$. There exists a small dimensional constant $\mu=\mu(4)<1$, such that if 
      \begin{equation}
          |Du|\leq \mu \quad \text{in}\quad B_3,
      \end{equation}
      then the following doubling  inequality holds: 
      \begin{equation}
          \sup_{B_2} H \leq C\left( 4, \|u\|_{L^{\infty}(B_3)}, \sqrt{f_0}, f_{0}^{-1} \right) \sup_{B_1}H.
      \end{equation}
  \end{proposition}

    \begin{proof}
        Consider the following test function on $B_3$:
        \[ P=2\log\rho + \alpha\left( x\cdot Du -u \right)  + \dfrac{\beta}{2}W^2 + \log \max \{ \overline{b}, \gamma \},  \]
        where $\rho=9-|x|^2$, $W=\sqrt{1+|Du|^2}$, and $\overline{b}=b-\sup_{B_1}b$ for $b=\log H$. Small constants $\alpha,\beta$ and a large constant $\gamma$ will be chosen later. Denote $\Gamma= \|u\|_{L^{\infty}(B_3)}+2 $. Throughout this proof, $C=C(4)$ denotes the dimensional constant which may change line by line.

        Suppose that $P$ attains its interior maximum at $x_0\in B_3$. Note that $|Du|\leq \mu<1 $ on $B_3$, which implies $1\leq W\leq \sqrt{2}$. If $|x_0|\leq 1$, then at $x_0$ we have,
        \begin{equation}\label{eq: upper bound on P-1}
            P\leq C+ \alpha\Gamma+\beta +\gamma.
        \end{equation}
        Therefore, we may assume that $1<|x_0|<3$. If $\overline{b}(x_0)\leq \gamma$, we again get \eqref{eq: upper bound on P-1}, so we further assume that $\overline{b}(x_0)\ge \gamma$ is sufficiently large.

        By appropriately choosing the tangent frame $\{e_1,\cdots, e_4\}$, we may assume that the second fundamental form $(h_{ij})$ is diagonal at the point $X_0=(x_0, u(x_0))$ with $h_{ii}=\kappa_i$ and $\kappa_1\ge\cdots\ge \kappa_4$. Consequently, the linearized operator $(F^{ij})$ is also diagonal with $F^{ii}=H-\kappa_i$. Hereafter, all computations are carried out at the point $X_0$.

        Recalling that the position vector on $\Sigma$ is $X=(x,u(x))$ and the outer (downward) unit normal to $\Sigma$ is $\nu=\frac{1}{W}(Du,-1)$, we can view the test function $P$ as a function on $\Sigma$ via the following geometric relations:
        \begin{equation*}
            \begin{split}
                \rho= 9-|X|^2+ \innerproduct{X,E_{n+1}}^2,\quad W=-\frac{1}{\innerproduct{\nu, E_{n+1}}}, \quad \text{and}\quad x\cdot Du - u= W\innerproduct{X,\nu}.
            \end{split}
        \end{equation*}
        Here the dimension $n=4$, and $E_{n+1}$ denotes the vertical direction of $\bb{R}^{n+1}$.

        Since $X_0$ is the maximum point, we have at $X_0$
        \begin{equation}\label{eq: DP=0}
            0=P_i= \dfrac{2\rho_i}{\rho} + \alpha\left( W\innerproduct{X,\nu} \right)_i + \beta WW_i + \dfrac{\overline{b}_i}{\overline{b}},
        \end{equation}
        and
        \begin{equation}\label{eq: negative ineq-1}
            \begin{split}
                0\ge \sum_i F^{ii}P_{ii} =&  \,2\dfrac{\sum_{i}F^{ii}\rho_{ii}}{\rho} - \dfrac{\sum_i F^{ii}\rho_{i}^{2}}{\rho^2} + \alpha \sum_{i} F^{ii}\left(W\innerproduct{X,\nu}\right)_{ii} \\
                &+ \beta \sum_{i} F^{ii}W_{i}^2 + \beta W \sum_{i} F^{ii}W_{ii}
                +\dfrac{\sum_{i} F^{ii} \overline{b}_{ii} }{\overline{b}} - \dfrac{\sum_{i} F^{ii}\overline{b}_{i}^{2} }{\overline{b}^2}.
            \end{split}
        \end{equation}

        We first treat the terms involving the cutoff function $\rho$. Computing its covariant derivatives, we have
        \begin{equation*}
            \rho_i = -2 \innerproduct{X,e_i} + 2\innerproduct{X,E_{n+1}}\innerproduct{e_i, E_{n+1}}:= -2a_i,
        \end{equation*}
        where 
        \begin{equation*}
            a_i = \innerproduct{X,e_i} - \innerproduct{X,E_{n+1}}\innerproduct{e_i, E_{n+1}} = \sum_{k=1}^{4}\innerproduct{X, E_k}\innerproduct{e_{i}, E_{k}}= \sum_{k=1}^{4} x_k \innerproduct{e_i, E_k}.
        \end{equation*}
        This implies that $|a_i|\leq C$, and hence
        \begin{equation}\label{eq: est on rho_i}
            \sum_{i} F^{ii}\rho_i^2 \leq C\sum_{i} F^{ii}\leq CH.
        \end{equation}
        Taking the covariant derivative on $\rho$ again and using the Gauss formula and Weingarten equation, we obtain
        \begin{equation*}
            \begin{split}
                \rho_{ii}&= -2 + 2\innerproduct{X,\nu}\kappa_i + 2\innerproduct{e_i, E_{n+1}}^2 - 2\innerproduct{X,E_{n+1}}\innerproduct{\nu, E_{n+1}} \kappa_i\\
            &= -2 + 2\innerproduct{e_i, E_{n+1}}^2 + 2\kappa_i \dfrac{x\cdot Du}{W}.
            \end{split}
        \end{equation*}
        Therefore, using the identity $\sum_{i}F^{ii}\kappa_i=2\sigma_2=2f_0>0$, we conclude that
        \begin{equation}\label{eq: est on rho_ii}
            \sum_{i}F^{ii}\rho_{ii}  \ge -C(H+f_0).
        \end{equation}

    Next, we compute the derivatives of $W$ and $x\cdot Du-u$. By a straightforward computation and using the linearized equation $\sum_i F^{ii}h_{iik}=0$, we have
    \begin{equation}\label{eq: deravaties of W}
        W_i = W^2 \innerproduct{e_i, E_{n+1}} \kappa_i \quad \text{and}\quad \sum_{i} F^{ii}W_{ii} = \dfrac{2}{W} \sum_{i} F^{ii}W_{i}^2 + W\sum_{i} F^{ii}\kappa_i^2.
    \end{equation}
    Furthermore, 
    \begin{equation}\label{eq: deravatives of W<X,nu>-1}
        (W\innerproduct{X,\nu})_{i} = W\left(  W\innerproduct{e_i, E_{n+1}}\innerproduct{X,\nu} + \innerproduct{X,e_i}  \right) \kappa_i, 
    \end{equation}
    and
    \begin{equation}\label{eq: deravatives of W<X,nu>-2}
         \sum_{i} F^{ii}(W\innerproduct{X,\nu})_{ii}> \dfrac{2}{W}\sum_{i}F^{ii}W_i(W\innerproduct{X,\nu})_i.
    \end{equation}
    Denote $d_i= W\innerproduct{e_i, E_{n+1}}\innerproduct{X,\nu} + \innerproduct{X,e_i}$. The following two claims establish key relations between $a_i$ and $b_i$. We postpone their proofs until the end of this section to avoid interrupting the main argument.

    \begin{claim}\label{CLAIM: claim1}
        There holds $\sum_{i=1}^{4} a_{i}^2 = |x|^2 - \left(\frac{x\cdot Du}{W}\right)^2 $. 
    \end{claim}
    From this claim,  we deduce that $\frac{|x|^2}{2}\leq \sum_i a_i^2\leq |x|^2 $. Thus at the point $x_0$, we have $\frac{1}{2}\leq \sum_i a_i^2\leq 9$.

    \begin{claim}\label{CLAIM: claim2}
        For any $i=1,2,3,4$, we have the following relation:
        \begin{equation}
            d_i = a_i + W^2 \sum_{k=1}^{4} a_k \innerproduct{e_k, E_{n+1}}\innerproduct{e_i, E_{n+1}}. 
        \end{equation}
    \end{claim}
    Note that each $|\innerproduct{e_i, E_{n+1}}|\leq  |E_{n+1}^{\top}|=\sqrt{1- \innerproduct{\nu, E_{n+1}}^2}= \frac{|Du|}{W}\leq \frac{\mu}{W}$, thus Claim \ref{CLAIM: claim2} gives a rough bound on $d_i$:
    \[  |d_i|\leq |a_i|+\mu^2 \sum_{k}|a_k|\leq C.  \]

    Inserting \eqref{eq: est on rho_i}, \eqref{eq: est on rho_ii}, \eqref{eq: deravaties of W}, \eqref{eq: deravatives of W<X,nu>-1} and \eqref{eq: deravatives of W<X,nu>-2} into \eqref{eq: negative ineq-1}, and combining with the almost Jacobi inequality:
    \begin{equation*}
        \sum_{i}F^{ii}\overline{b}_{ii}\ge \dfrac{2}{9}\left( \dfrac{\kappa_4}{H} + \dfrac{1}{2} \right) \sum_{i} F^{ii} \overline{b}_{i}^2 - \sum_{i} F^{ii}\kappa_{i}^2,
    \end{equation*}
    we obtain
    \begin{equation*}
        \begin{split}
            0\ge &-C\dfrac{H+f_0}{\rho^2} + \boxed{ 2\alpha W^2 \sum_{i} F^{ii} d_{i} \innerproduct{e_i, E_{n+1}}\kappa_i^2} + 3\beta\sum_{i} F^{ii}W_{i}^2 \\
            &+ \left( \beta- \dfrac{1}{\overline{b}} \right)\sum_{i} F^{ii}\kappa_{i}^2 + \dfrac{2}{9}\left( \dfrac{\kappa_4}{H} +\dfrac{1}{2}\right) \sum_{i} F^{ii}\dfrac{\overline{b}_i^2}{\overline{b}} - \sum_{i} F^{ii}\dfrac{\overline{b}_i^2}{\overline{b}^2}.
        \end{split}
    \end{equation*}
    The boxed term is a bad term which does not have a definite sign. To control this term, we impose the condition that $|\innerproduct{e_i, E_{n+1}}|\leq \frac{|Du|}{W}\leq \mu$ to be sufficiently small. Using this smallness assumption on the gradient, the above inequality becomes
    \begin{equation}\label{eq: negative ineq-2}
        \begin{split}
            0\ge &-C\dfrac{H+f_0}{\rho^2} + 3\beta\sum_{i}F^{ii}W_{i}^2 + \left( \beta-C\mu\alpha-\dfrac{1}{\overline{b}} \right) \sum_{i} F^{ii}\kappa_{i}^2 \\
            &+ \dfrac{2}{9}\left( \dfrac{\kappa_4}{H} + \dfrac{1}{2} \right)\sum_{i} F^{ii}\dfrac{\overline{b}_{i}^{2}}{\overline{b}}  -  \sum_{i} F^{ii} \dfrac{\overline{b}_{i}^{2}}{\overline{b}^2}.
        \end{split}
    \end{equation}
    If $1/\overline{b} \ge \beta /4 $, we immediately obtain the desired bound on $\overline{b}$. Therefore, we may assume that $1/\overline{b}< \beta/4$. We further require
    \begin{equation}\label{eq: condition on parameters-1}
        C\mu\alpha < \dfrac{\beta}{4}.
    \end{equation}
    Substituting these conditions into \eqref{eq: negative ineq-2}, we get
    \begin{equation}
        0\ge -C\dfrac{H+f_0}{\rho^2} + 3\beta\sum_{i}F^{ii}W_{i}^2 + \dfrac{\beta}{2}\sum_{i}F^{ii}\kappa_{i}^2 + \dfrac{2}{9}\left( \dfrac{\kappa_4}{H} + \dfrac{1}{2} \right) \sum_i F^{ii}\dfrac{\overline{b}_i^2}{\overline{b}}- \sum_{i}F^{ii}\dfrac{\overline{b}_{i}^{2}}{\overline{b}^2}.
    \end{equation}
    If the nonnegative coefficient $\frac{2}{9}\left(\frac{\kappa_4}{H}+\frac{1}{2}\right)$ has a positive lower bound, the almost Jacobi inequality now becomes the full Jacobi inequality, and we may follow Guan-Qiu's argument in \cite{Guan-Qiu} with slight modification to derive the desired estimates. In the other case, $\kappa_4/H$ is close to $-1/2$, so the almost Jacobi inequality degenerates. As compensation, now each $|\kappa_{i}|$ is comparable with $H$, which means that the equation becomes conformally uniformly elliptic. Thus the term $\sum_{i}F^{ii}\kappa_i^2$ provides the required positive contribution.

    \vspace{0.3cm}
    \emph{Case 1. Degenerate Case: $-1/2\leq \kappa_4/H\leq -1/4$.}

    From the $DP=0$ equation \eqref{eq: DP=0}, we obtain
    \begin{equation*}
        \left( \dfrac{\overline{b}_i}{\overline{b}} \right)^2 = \left( \dfrac{2\rho_i}{\rho} + \alpha W d_i \kappa_i + \beta WW_i \right)^2 \leq \dfrac{C}{\rho^2} + C\alpha^2\kappa_i^2 + 6\beta^2 W_{i}^2.
    \end{equation*}
    Substituting this into \eqref{eq: negative ineq-2}, we get
    \begin{equation*}
        0\ge -C\dfrac{H+f_0}{\rho^2} + (3\beta-6\beta^2) \sum_{i}F^{ii}W_{i}^2 + \left( \dfrac{\beta}{2} -C\alpha^2 \right) \sum_{i}F^{ii}\kappa_{i}^2.
    \end{equation*}
    We choose
    \begin{equation}\label{eq: condition on parameters-2}
        \beta\leq 1/2 \quad \text{and}\quad C\alpha^2\leq \dfrac{\beta}{4},
    \end{equation}
    and moreover, note that $F^{44}=H-\kappa_4\ge H$. The above inequality implies the following estimate:
    \begin{equation*}
        \rho^2 \overline{b} \leq \rho^2 H^2\leq \dfrac{C}{\beta} \left( 1+ \dfrac{f_0}{H} \right)\leq \dfrac{C}{\beta} (1+\sqrt{f_0}),
    \end{equation*}
    where we used $H=\sqrt{2f_0+|A^2|}\ge \sqrt{2f_0}$ in the last step. This gives an upper bound on $P$ at $x_0$:
    \begin{equation}\label{eq: upper bound on P-2}
        P\leq C\Gamma + \log \beta^{-1}+ \log\left(1+\sqrt{f_0}\right). 
    \end{equation}

    \vspace{0.3cm}
    \emph{Case II. Non-degenerate Case: $\kappa_4/H\ge -1/4$.}
    
    In this case, we choose $\overline{b}\ge \gamma > 36$, so \eqref{eq: negative ineq-2} reduces to
    \begin{equation}\label{eq: negative ineq-3}
        0\ge -C\dfrac{H+f_0}{\rho^2} +3\beta \sum_{i}F^{ii}W_{i}^2 + \dfrac{\beta}{2} \sum_{i} F^{ii}\kappa_i^2 + \dfrac{\overline{b}}{36}\sum_{i}F^{ii}\dfrac{\overline{b}_i^2}{\overline{b}^2}.
    \end{equation}

    \vspace{0.3cm}
    \emph{Subcase II-1: $a_{k}^2\ge 1/8$ for some $k\ge2$.}

    Using the $DP=0$ equation \eqref{eq: DP=0} again and recalling that $\rho_i=-2a_i$, we know that
    \begin{equation*}
        \left( \dfrac{\overline{b}_{k}^2}{\overline{b}} \right)^2 = \left(  \dfrac{-4a_k}{\rho} +\alpha W d_k\kappa_k +\beta WW_{k}  \right)\ge \dfrac{8a_k^2}{\rho^2} - C\alpha^2\kappa_k^2 - C\beta^2 W_{k}^2.
    \end{equation*}
    Inserting it into \eqref{eq: negative ineq-3} and recalling from Lemma \ref{LEMMA: est on linearized operator} that $F^{kk}=H-\kappa_k \ge \left(1- \frac{1}{\sqrt{2}}\right)H\ge H/10 $, we conclude that
    \begin{align*}
        0&\ge -C\dfrac{H+f_0}{\rho^2} + 3\beta F^{kk}W_{k}^2 + \dfrac{\beta}{2}F^{kk}\kappa_k^2 + \dfrac{\gamma}{36}F^{kk}\left( \dfrac{1}{\rho^2}-C\alpha^2\kappa_k^2-C\beta^2\kappa_k^2 W_k^2 \right)\\
        &\ge \dfrac{H}{\rho^2}\left( \dfrac{\gamma}{360} - C- \dfrac{C\rho f_0}{H} \right) + \beta \left( 3 - \dfrac{C\gamma}{36}\beta \right) F^{kk}W_{k}^2 + \left(\dfrac{\beta}{4}- \dfrac{C\gamma}{36}\alpha^2\right)F^{kk}\kappa_k^2.
    \end{align*}
    By choosing 
    \begin{equation}\label{eq: condition on parameters-3}
        \gamma\ge C(1+\sqrt{f_0})\quad  \text{for}\ C=C(4)\ \text{sufficiently large,} 
    \end{equation}
    and
    \begin{equation}\label{eq: condition on parameters-4}
        \beta< \dfrac{1}{C\gamma},\qquad \alpha^2< \dfrac{1}{C\gamma}\beta,
    \end{equation}
    we obtain a contradiction.
    
    \vspace{0.3cm}
    \emph{Subcase II-2: $a_1^2\ge1/8$.}
    From Claim \ref{CLAIM: claim2}, we know that
    \begin{equation*}
        |d_1|= \left| a_1 + W^2 \sum_{k=1}^{4} a_{k} \innerproduct{e_k,E_{n+1}}\innerproduct{e_{1}, E_{n+1}} \right| \ge |a_1| - \mu^2\sum_{k=1}^{4}|a_k|\ge \dfrac{1}{2\sqrt{2}} - 12\mu^2.
    \end{equation*}
    By taking $\mu$ sufficiently small, we obtain $|d_1|\ge 1/4$. The $DP=0$ equation \eqref{eq: DP=0} then implies that
    \begin{equation*}
        \begin{split}
            \left(\dfrac{\overline{b}_1}{\overline{b}}\right)^2 &= \left( \dfrac{2\rho_1}{\rho} + \alpha W d_1\kappa_1 +\beta W^3\innerproduct{e_1,E_{n+1}}\kappa_1 \right)^2 \\
            &\ge \dfrac{1}{2}W^2\kappa_1^2\left( \alpha d_1 + \beta W^2\innerproduct{e_1, E_{n+1}} \right)^2 - \dfrac{C}{\rho^2}.
        \end{split}
    \end{equation*}
    We further require 
    \begin{equation}\label{eq: condition on parameters-5}
        \dfrac{\alpha}{8}\ge 2\beta,
    \end{equation}
    which gives
    \begin{equation*}
         \left(\dfrac{\overline{b}_1}{\overline{b}}\right)^2\ge \dfrac{\alpha^2}{128}\kappa_1^2-\dfrac{C}{\rho^2}.
    \end{equation*}
    If $C/\rho^2 \ge \alpha^2\kappa_1^2/256$, then $\rho^2\overline{b}\leq \rho^2H^2\leq C\rho^2\kappa_1^2\leq C\alpha^{-2}$, which yields a desired estimate. Therefore, we may assume that $C/\rho^2\leq \alpha^2\kappa_1^2/256$, hence $\overline{b}_1^2/\overline{b}^2 \ge \alpha^2 H^2/C. $ Since $F^{11}=H-\kappa_1\ge f_0/H$ by Lemma \ref{LEMMA: est on linearized operator}, we conclude from \eqref{eq: negative ineq-3} that
    \begin{equation*}
        0\ge -C\dfrac{H+f_0}{\rho^2}+ \dfrac{\overline{b}}{36} F^{11}\dfrac{\overline{b}_1^2}{\overline{b}^2} \ge -C\dfrac{H+f_0}{\rho^2} + \dfrac{\alpha
        ^2}{C} f_0 H\overline{b}.
    \end{equation*}
    This implies that 
    \begin{equation*}
        \rho^2\overline{b}\leq \dfrac{C}{\alpha^2}\left(1+\dfrac{1}{f_0}\right).
    \end{equation*}
    Consequently, we obtain the following upper bound on $P$ at $x_0$:
    \begin{equation}\label{eq: upper bound on P-3}
        P\leq C\Gamma + \log\alpha^{-2} + \log \left(1+\dfrac{1}{f_0}\right)    \end{equation}

    Finally, for a sufficiently large constant $C=C(4)$, we take
    \begin{equation*}
        \gamma= C(1+\sqrt{f_0}),\quad \beta=\dfrac{1}{256C\gamma},\quad \alpha=\dfrac{1}{16C\gamma}, \quad \text{and}\quad \mu=\dfrac{1}{128C},
    \end{equation*}
    which satisfies all the parameter constraints: \eqref{eq: condition on parameters-1}, \eqref{eq: condition on parameters-2}, \eqref{eq: condition on parameters-3}, \eqref{eq: condition on parameters-4}, and \eqref{eq: condition on parameters-5}. We emphasize that the constant $\mu$ in the gradient smallness assumption  depends only on the dimension. 
    
    In conclusion, combining the upper bounds \eqref{eq: upper bound on P-1}, \eqref{eq: upper bound on P-2}, and \eqref{eq: upper bound on P-3}, we obtain
    \begin{equation*}
        \sup_{B_2} P \leq P(x_0) \leq C\left(4, \|u\|_{L^{\infty}}, \sqrt{f_0}, f_0^{-1} \right).
    \end{equation*}
    By the definition of $P$ and $\overline{b}$, we finally obtain the doubling inequality:
    \begin{equation*}
        \dfrac{\sup_{B_2} H}{\sup_{B_{1}} H}\leq C\left(4, \|u\|_{L^{\infty}}, \sqrt{f_0}, f_0^{-1} \right).
    \end{equation*}
    \end{proof}

    We now give the proofs of the two claims used previously.

    \vspace{0.3cm}
    \noindent\emph{Proof of Claim \ref{CLAIM: claim1}.} This claim follows from a straightforward computation. Recalling that
    \[  a_i = \innerproduct{X,e_i} - \innerproduct{X,E_{n+1}}\innerproduct{e_i, E_{n+1}},  \]
    and combining this with $X=(x,u)$ and $\nu=\frac{1}{W}(Du,-1)$, we obtain
    \begin{align*}
        \sum_{i=1}^{4}a_{i}^2&= \sum_{i=1}^{4} \innerproduct{X, e_i}^2 + \innerproduct{X, E_{n+1}}^2\sum_{i=1}^{4}\innerproduct{e_i, E_{n+1}}^2 - 2\innerproduct{X, E_{n+1}} \sum_{i=1}^{4}\innerproduct{X,e_i}\innerproduct{e_i, E_{n+1}}\\
        &= |X|^2 - \innerproduct{X,\nu}^2 + u^2 (1- \innerproduct{\nu, E_{n+1}}^2 ) - 2u \left( u - \innerproduct{X,\nu}\innerproduct{\nu,E_{n+1}} \right)\\
        &= \left(|X|^2-u^2\right) - \left( \innerproduct{X,\nu}- u \innerproduct{\nu, E_{n+1}} \right)^2\\
        &= |x|^2 - \left( \dfrac{x\cdot Du-u}{W}+\dfrac{u}{W} \right)^2\\
        &= |x|^2 - \left( \dfrac{x\cdot Du}{W} \right)^2.
    \end{align*}
    \hfill\qed

    \vspace{0.3cm}
    \noindent\emph{Proof of Claim \ref{CLAIM: claim2}.} Recalling the definition of $b_i$,
    \[ d_i= W\innerproduct{X,\nu}\innerproduct{e_i, E_{n+1}} + \innerproduct{X,e_i}. \]
    From the definition of $a_i$, we have $\innerproduct{X,e_i}= a_i+ u\innerproduct{e_i, E_{n+1}} $, so substituting this into above gives
    \[  d_i= (x\cdot Du - u) \innerproduct{e_i, E_{n+1}}+ a_i + u\innerproduct{e_i, E_{n+1}}= a_i + (x\cdot Du)\innerproduct{e_i, E_{n+1}}.  \]
    Next, we compute the term $(x\cdot Du)\innerproduct{e_i, E_{n+1}}$. First, note that $x\cdot Du = \sum_{k=1}^{4}\innerproduct{X,E_k}\innerproduct{E_k,\nu}W$. We also need the following two orthonormal projection identities: 
    \[   \innerproduct{E_k,\nu}\nu = E_k- \sum_{s=1}^{4} \innerproduct{E_k, e_s} e_s\quad \text{and}\quad \innerproduct{e_i, E_{n+1}} E_{n+1} = e_i - \sum_{t=1}^{4} \innerproduct{e_i, E_t} E_t.   \]
    Taking the inner product of these two identities and rearranging terms, we obtain
    \[  -\dfrac{1}{W}\innerproduct{E_k,\nu} \innerproduct{e_i, E_{n+1}} = \sum_{s,t=1}^{4} \innerproduct{E_k, e_s}\innerproduct{e_s, E_t} \innerproduct{E_t, e_i} - \innerproduct{E_k, e_i}.  \]
    Therefore, substituting back into the expression for $(x\cdot Du) \innerproduct{e_i, E_{n+1}}$, we obtain
    \begin{align*}
        (x\cdot Du) \innerproduct{e_i, E_{n+1}} &= W \sum_{k=1}^{4} \innerproduct{X,E_k}\innerproduct{E_k, \nu}\innerproduct{e_i, E_{n+1}}\\
        &= W^2 \sum_{k=1}^{4} \innerproduct{X,E_k}\left( \innerproduct{E_k,e_i} - \sum_{s,t=1}^{4} \innerproduct{E_k, e_s} \innerproduct{e_s, E_t}\innerproduct{E_t, e_i} \right)\\
        &= W^2 a_i - W^2 \sum_{s,t=1}^{4} a_s \innerproduct{e_s, E_t}\innerproduct{E_t, e_i}\\
        &= W^2 a_i - W^2 \sum_{s=1}^{4} a_s \left( \innerproduct{e_s,e_i} - \innerproduct{e_s, E_{n+1}}\innerproduct{e_i, E_{n+1}} \right)\\
        &= W^2\sum_{s=1}^{4} a_{s} \innerproduct{e_s, E_{n+1}} \innerproduct{e_i, E_{n+1}}.
    \end{align*}
    Finally, we conclude that
    \[  d_i = a_i + W^2 \sum_{s=1}^{4} a_{s} \innerproduct{e_s, E_{n+1}} \innerproduct{e_i, E_{n+1}}, \]
    which completes the proof of Claim \ref{CLAIM: claim2}. \hfill\qed

    By a rescaling argument, we obtain the following rescaled version of the doubling inequality, which asserts that the curvature at larger scales can be controlled by the curvature on any smaller ball.

    \begin{corollary}\label{COR: rescaled doubling}
        Let $\Sigma=(x,u(x))$ be a smooth $2$-convex graph over $B_1\subset \bb{R}^4$ satisfying the scalar curvature equation
      \begin{equation*}
          \sigma_{2}(\kappa) = 1 \quad \text{on}\quad B_1,
      \end{equation*}
       There exists a small dimensional constant $\mu=\mu(4)<1$, such that if 
      \begin{equation*}
          |Du|\leq \mu \quad \text{in}\quad B_{3r},
      \end{equation*}
      for some radius $0<r<1/3$, then for any $y\in B_{r/3}$ and $0<\rho< 2r/3$, we have
      \begin{equation}
          \sup_{B_r} H \leq C\left( 4, \|u\|_{L^{\infty}(B_1)}, r, r^{-1}, \rho,\rho^{-1} \right) \sup_{B_{\rho}(y)}H.
      \end{equation}
    \end{corollary}

\section{Alexandrov regularity}\label{Section: Alexandrov Regularity}

In this section, we establish the Alexandrov regularity for viscosity solutions to the scalar curvature equation. Our approach adapts the proof framework of the classical Alexandrov theorem for convex functions, as presented in \cite[Section 6.4]{Evans-Gariepy}. The proof of the classical Alexandrov theorem relies on two crucial ingredients: the interpretation of the Hessian as Radon measures, and gradient (or Lipschitz) estimates for convex functions. The former can be heuristically understood as convex functions admitting the following a priori $W^{2,1}$ estimates:
\begin{equation}\label{eq: 5.1}
    \int_{B_{1/2}}|D^2u|   \leq  \int_{B_{1/2}} \Delta u  = -\int_{\partial B_{1/2}} \dfrac{\partial u}{\partial \nu} \leq  C(n) \|Du\|_{L^{\infty}(B_{1/2})}\leq C(n)\|u\|_{L^{\infty}(B_1)}.
\end{equation}
Motivated by this heuristic, we are essentially able to establish the Alexandrov regularity under more general structural hypotheses on the function.

\begin{proposition}\label{PROP: general Alexandrov}
    Let $u$ be a continuous function on $B_{1}$ satisfying the following hypotheses:

    $(\mathrm{H}1)$ (Almost everywhere differentiability) $u$ is differentiable almost everywhere, and for each differentiable point $x\in B_1$, there holds
    \begin{equation}\label{eq: H1}
        \lim_{r\to 0} \fint_{B_{r}(x)}|Du(y)-Du(x)| \,\mathrm{d}y =0 .
    \end{equation}

    $(\mathrm{H}2)$ (Hessian as Radon measures) The distributional Hessian of $u$ can be interpreted as Radon measures. That is, there exist Radon measures $\left(\mu^{ij} \right)_{i,j=1,\cdots,n} $ such that $\mu^{ij}=\mu^{ji}$ and 
    \begin{equation}
        \int u\, \partial^{2}_{ij}\phi  = \int \phi \,\mathrm{d}\mu^{ij}\quad \text{for any}\ \phi\in C_{c}^{\infty}(B_1).
    \end{equation}

    $(\mathrm{H}3)$ (Lipschitz estimates) For almost every $x\in B_1$, let $g(y)= u(y)-u(x)- Du(x)\cdot (y-x)$. There exists a constant $r_x>0$, such that the following Lipschitz estimate holds for $0<r<r_x:$
    \begin{equation}\label{eq: H3}
        \sup_{\substack{y,z\in B_{r}(x)\\ y\neq z}} d_{y,z}^{n+1}\dfrac{|g(y)-g(z)|}{|y-z|}\leq C\int_{B_{r}(x)}|g(y)|\,\mathrm{d}y,
    \end{equation}
    where $d_{y,z}=\min\{d_y,d_z\}$ for $d_y=\mathrm{dist}(y,\partial B_{r}(x))$, and $C$  is a constant independent of $r$.

    Then, $u$ is twice differentiable almost everywhere in $B_{1}$. That is, for almost every $x\in B_1$, there exists a quadratic polynomial $Q$ such that
    \[ |u(y)-Q(y)|= o\left( |y-x|^2 \right),\quad \text{as}\ y\to x. \]
\end{proposition}
\begin{proof}
    \emph{Step 1. Approximation in $L^1$ sense.} By the Lebesgue-Radon-Nikodym decomposition, we write $\mu^{ij}= u^{ij}\,\mathrm{d}x + \mu^{ij}_{s}$, where $\mathrm{d}x$ denotes the $n$-dimensional Lebesgue measure, $u^{ij}\in L^{1}_{\mathrm{loc}}$ denotes the absolutely continuous part with respect to $\mathrm{d}x$, and $\mu^{ij}_{s}$ denotes the singular part. Write $[D^2u]= D^2u \,\mathrm{d}x + [D^2u]_{s}$ for $D^2u=\left(u^{ij}\right)$ and $[D^2u]_{s}= [\mu^{ij}_{s}]$. For almost every $x\in B_1$, there hold
    \begin{align}
    &\lim_{r\to 0}\fint_{B_{r}(x)}|D^2u(y)-D^2u(x)|\,\mathrm{d}y=0, \label{eq: 5.5} \\
    &\lim_{r\to 0}\dfrac{1}{r^n}\|[D^2u]_{s}\|(B_{r}(x))=0. \label{eq: 5.6}
\end{align}

    Fix any $x$ such that \eqref{eq: H1}, \eqref{eq: H3}, \eqref{eq: 5.5} and \eqref{eq: 5.6}  hold, following verbatim the argument in  Steps 2-4 of \cite[pp. 274-275]{Evans-Gariepy}, we can conclude that
    \[   \fint_{B_{r}(x)}|h(y)|\,\mathrm{d}y = o(r^2)\quad \text{as}\ r\to 0,  \]
    where $h(y)= u(y)- u(x) - Du(x)\cdot (y-x) - \frac{1}{2}(y-x)^{T}D^2u(x)(y-x)$.
    
    \vspace{0.3cm}
    \emph{Step 2. Lipschitz estimates for $h$.}  For any $0<2r<r_x$, using the Lipschitz estimate \eqref{eq: H3} for $g$, we have
   \begin{equation}\label{eq: 5.7}
       \begin{split}
           r^{n+1} \sup_{\substack{y,z\in B_{r}(x)\\ y\neq z}} \dfrac{|g(y)-g(z)|}{|y-z|} &\leq \sup_{\substack{y,z\in B_{2r}(x)\\ y\neq z}} d_{y,z}^{n+1} \dfrac{|g(y)-g(z)|}{|y-z|}\\
           &\leq C \int_{B_{2r}(x)} |g(y)|\,\mathrm{d}y\\
           &\leq  C\left( \int_{B_{2r}(x)} |h(y)|\,\mathrm{d}y + r^{n+2} \right),
       \end{split}
   \end{equation}
   where $d_{y,z}=\min\{ 2r-|y-x|, 2r-|z-x| \}$ and $C$ is independent of $r$. Combining \eqref{eq: 5.7} with the identity
   \[  (y-x)^{T}D^2u(x)(y-x)-(z-x)^{T}D^2u(x)(z-x)= (y+z-2x)^{T}D^2u(x)(y-z),  \]
   we obtain the following Lipschitz estimate for $h$:
   \begin{equation}
      \sup_{\substack{y,z\in B_{r}(x)\\ y\neq z}} \dfrac{|h(y)-h(z)|}{|y-z|} \leq \sup_{\substack{y,z\in B_{r}(x)\\ y\neq z}} \dfrac{|g(y)-g(z)|}{|y-z|} + Cr
       \leq \dfrac{C}{r}\fint_{B_{2r}(x)}|h(y)|\,\mathrm{d}y +Cr.
   \end{equation}

    \vspace{0.3cm}
    \emph{Step 3: Improve $L^1$ approximation to $L^{\infty}$.} For any small $\varepsilon>0$, we may find a sufficiently small $r_0>0$ such that
    \[ \dfrac{1}{r^2}\sup_{B_{r/2}(x)}|h| \leq 2\varepsilon,\quad \text{for}\ r<r_0.  \]
    Fix a small parameter $0<\eta<1/2$ to be chosen later. By Step 1, we have
    \begin{equation}\label{eq: 5.9}
        \left| \{ z\in B_{r}(x): |h(z)|\ge\varepsilon r^2 \}  \right| \leq \dfrac{1}{\varepsilon r^2} \int_{B_{r}(x)} |h(z)|\,\mathrm{d}z = \dfrac{o(r^{n+2})}{\varepsilon r^2} < \dfrac{1}{2}\eta^n |B_{r}|,
    \end{equation}
    provided $r< r_0= r_0(n,\eta,\varepsilon)$. We claim that for any $y\in B_{r/2}(x)$, there exists $z\in B_{\eta r}(y)$ such that $|h(z)|\leq \varepsilon r^2$. Otherwise, we would have $B_{\eta r}(y)\subset \{|h|\ge \varepsilon r^2\} \cap B_{r}(x)$, which implies that
    \[  \eta^n|B_r|\leq \left| \{ z\in B_{r}(x): |h(z)|\ge\varepsilon r^2 \}  \right| \leq \dfrac{1}{2} \eta^n |B_r|. \]
    This contradicts \eqref{eq: 5.9}. Therefore, 
    \begin{equation*}
        \begin{split}
            |h(y)|\leq |h(z)|+ \eta r \dfrac{|h(y)-h(z)|}{|y-z|} &\leq \varepsilon r^2 + \eta r\left( \dfrac{C}{r}\fint_{B_{2r}(x)}|h| + Cr \right)\\
            &\leq \varepsilon r^2 + \eta o(r^2) + C\eta r^2.
        \end{split}
    \end{equation*}
    We choose $\eta = \varepsilon/2C$ and reduce $r_0$ if needed, then we finally conclude that
    \[  |h(y)| \leq 2\varepsilon r^2 \quad \text{for}\ r<r_0=r_0(n,\varepsilon).  \]
    Since $y\in B_{r/2}(x)$ is arbitrary, this completes the proof.

\end{proof}

 For smooth $2$-convex functions, we have an a priori $W^{2,1}$ estimate analogous to \eqref{eq: 5.1}, since $\Delta u= \sqrt{2\sigma_2(D^2u) + |D^2u|^2}$. It follows that $(\mathrm{H}2)$ holds for general $2$-convex functions via a standard approximation argument.
 
 For $k$-convex functions with $k>n/2$, while there is no gradient (or Lipschitz) estimate, it was proved by Trudinger and Wang \cite{Trudinger-Wang} that they admit $W^{1,n+}$ and H\"older estimates. As a consequence, $(\mathrm{H}1)$ is satisfied, and the Lipschitz estimate in $(\mathrm{H}3)$ can be replaced by the corresponding H\"older estimates. The Alexandrov regularity for $k$-convex functions with $k>n/2$ can thus be established. It was first proved by Chaudhuri and Trudinger \cite{Chaudhuri-Trudinger}. 

In \cite{Shankar-Yuan-annals}, Shankar and Yuan studied the equation $\sigma_{2}(D^2u)=1$ in dimension 4. In this case, $k=2$ and $n=4$, which is exactly the critical case of Chaudhuri-Trudinger's result, so it cannot be directly applied. However,  they have a stronger tool --- the PDE $\sigma_{2}(D^2u)=1$. Earlier works by Trudinger \cite{Trudinger-CPDE}, and also Chou-Wang \cite{Chou-Wang}, established that $2$-convex solutions to $\sigma_{2}(D^2u)=1$ satisfy the required Lipschitz estimates, which ensures that $(\mathrm{H}1)$ and $(\mathrm{H}3)$ are satisfied. Therefore, Shankar and Yuan proved the Alexandrov regularity for viscosity solutions to $\sigma_{2}(D^2u)=1$. For the Alexandrov regularity of general $k$-convex viscosity solutions to $\sigma_{k}(D^2u)=f$ with positive Lipschitz right-hand side $f$, we refer the reader to \cite{Fan}.  

We now  turn our attention to the scalar curvature equation. First, Korevaar \cite{Korevaar} proved the following gradient estimates:

\begin{proposition}[Korevaar]\label{PROP: Korevaar grad est}
    Let $\Sigma=(x,u(x))$ be a smooth $2$-convex graph over $B_1\subset\bb{R}^n$ satisfying the scalar curvature equation
    \begin{equation}\label{eq: 5.10}
        \sigma_{2}(\kappa)=f_0\quad \text{on}\quad B_1,
    \end{equation}
    for some positive constant $f_0$. Then,
    \begin{equation}\label{eq: Korevaar grad est}
        \sup_{B_{1/2}}|Du|\leq C(n) \exp \left\{ C(n) \left(1+\sqrt{f_0}\right)  \left( \sup_{B_1} u - u(0) \right)^2  \right\}.
    \end{equation}
    
\end{proposition}   

\begin{remark}\label{RMK: Lip reg for viscosity solu}
    By solving the Dirichlet problem with smooth approximating boundary data \cite{Ivochkina, Ivochkina-Lin-Trudinger}, we deduce from the above gradient estimate that any viscosity solution $u$ to \eqref{eq: 5.10} is locally Lipschitz. There is a subtle technical issue to clarify here: when solving the Dirichlet problem for \eqref{eq: 5.10} on a domain $\Omega\subset B_1$, a curvature condition on the boundary $\partial\Omega$ is required (see \cite[Theorem 4.1]{Ivochkina}), namely
        \begin{equation*}
            \dfrac{n-2}{n}\sigma_2(\kappa_{\partial\Omega})>1.
        \end{equation*}
    To apply this Dirichlet solvability result rigorously, we may first solve the Dirichlet problem with smooth approximating boundary data on any small ball $B_{\rho}(y)\subset B_1$, where $\rho$ is sufficiently small such that the above curvature condition holds. This yields $u\in C^{0,1}_{\mathrm{loc}}\left(B_{\rho}(y)\right)$, and further conclude that $u\in C^{0,1}_{\mathrm{loc}}\left(B_{1}\right)$. By the Rademacher theorem and the Lebesgue differentiation theorem, $u$ is differentiable almost everywhere and satisfies \eqref{eq: H1}, which verifies hypothesis $(\mathrm{H}1)$.
\end{remark}

In the following two subsections, we will verify hypotheses $(\mathrm{H}2)$ and $(\mathrm{H}3)$ respectively.

\subsection{A priori $W^{2,1}$ estimates/ interpretation of the Hessian as Radon measures}

 By the gradient bound of the solution $u$ from Korevaar's estimate (Proposition \ref{PROP: Korevaar grad est}), the Hessian $D^2u$ is approximately equal to the normalized second fundamental form $A=g^{-1}h$ of the graph $\Sigma=(x,u(x))$ in the natural coordinates. Moreover, the $2$-convexity of $\Sigma$ further implies  that the normalized second fundamental form $A$ can be controlled by the mean curvature. Since the mean curvature operator has a divergence structure, we can derive the a priori $W^{2,1}$ estimate via integration by parts. More generally, this argument extends to yield $W^{2,1}$ estimates for any $2$-convex graph with an a priori gradient bound.

\begin{proposition}\label{PROP: W^2,1 est}
    Let $\Sigma=(x,u(x))$ be a smooth $2$-convex graph over $B_1\subset \bb{R}^n$. Suppose that $u$ has an a priori gradient bound $\|Du\|_{L^{\infty}(B_1)}=\Gamma$, then
    \[ \int_{B_1}|D^2u|\leq C= C(n,\Gamma). \]
\end{proposition}
\begin{proof}
    Under the natural coordinates of $\bb{R}^{n+1}$, the induced metric on $\Sigma$ is given by $g=(g_{ij})= (\delta_{ij}+ u_i u_j)$, and the second fundamental form is $h=(h_{ij})= \left(u_{ij}/W\right) $. Denote the normalized second fundamental form by $A=g^{-1}h=(a_{ij})$. From the $2$-convexity of $\Sigma$, we have
    \[ H=\sigma_{1}(A)>0\quad \text{and} \quad \sigma_2(A)=\dfrac{1}{2}\left(H^2-|A|^2\right)>0. \]
    It follows that $|a_{ij}|\leq H$ for all $i,j=1,\cdots, n$. Then
    \[ \int_{B_1}|u_{ij}| = \int_{B_1} \left| W\sum_{k=1}^{n} g_{ik}a_{kj} \right|\leq C(n,\Gamma)\int_{B_1} H. \]
    Using the divergence structure of $H$, we conclude that
    \[ \int_{B_1} H \,\mathrm{d}x = \int_{B_1} \mathrm{div}\left( \dfrac{Du}{W} \right)\,\mathrm{d}x = -\int_{\partial B_1} \dfrac{Du\cdot \nu_{\partial B_1}}{W}\,\mathrm{d}S \leq \mathrm{Area}(\partial B_1)\leq C(n).  \]

\end{proof}

\begin{proposition}\label{PROP: Hessian as Radon measures}
    Let $\Sigma_{k}=(x, u_k(x))$ be a sequence of smooth $2$-convex graphs over $B_1\subset\bb{R}^n$. Suppose that $\|Du_k\|_{L^{\infty}(B_1)}\leq \Gamma$, and $u_k\to u$ locally uniformly for some Lipschitz function $u$ on $B_1$. Then, in the distributional sense, the Hessian $D^2u$ can be interpreted as a matrix-valued Radon measure $[D^2u] = [\mu^{ij}]$ with $\mu^{ij}=\mu^{ji}$. That is, 
    \[  \int u \,\partial^2_{ij}\phi = \int \phi\,\mathrm{d}\mu^{ij} \quad \text{for any}\ \phi\in C_{c}^{\infty}(B_1).   \]
\end{proposition}
\begin{proof}
    For any $k\in\mathbb{N}^+$ and $i,j=1,\cdots,n$, define $\mu^{ij}_k= \partial^2_{ij} u_k \,\mathrm{d}x $. By Proposition \ref{PROP: W^2,1 est}, we have for any $\Omega\subset\subset B_1$ that $|\mu^{ij}_k |(\Omega)\leq C(n,\Gamma)$. By compactness, up to a subsequence, $\mu^{ij}_k$ converges weakly to a Radon measure $\mu^{ij}$. This is as our desired, since for any $\phi\in C_{c}^{\infty}(B_1)$, there holds
    \[  \int \phi \,\mathrm{d}\mu^{ij}=  \lim_k \int \phi \,\mathrm{d}\mu^{ij}_k = \lim_k \int \phi\, \partial^2_{ij} u_k = \lim_k \int u_k \,\partial^2_{ij}\phi  = \int u\,\partial^2_{ij}\phi. \]
\end{proof}

\subsection{Linearly dependent gradient estimates} \label{subsec: Linearly dependent gradient estimates}
    To establish the Lipschitz estimate \eqref{eq: H3}, we first need to establish the gradient estimate of the following form:
    \begin{equation}\label{eq: 5.12}
        \| Du \|_{L^{\infty}\left(B_{r/2}(x) \right)} \leq \dfrac{C}{r} \|u\|_{L^{\infty}\left(B_r(x)\right)}.
    \end{equation}
    Then, by interpolation, we can improve the dependence on the $L^{\infty}$ norm in \eqref{eq: 5.12} to the dependence on the $L^1$ norm.  We emphasize that for the interpolation to apply, the dependence on the $L^\infty$ norm on the right-hand side of estimate \eqref{eq: 5.12} must be linear. 

    Linearly dependent gradient estimates seem to be impossible for curvature equations.  For the minimal surface equation $\sigma_1(\kappa)=\mathrm{div}(Du/W)=0$, Bombieri, De Giorgi and Miranda \cite{B-DG-M} proved the following linear exponentially  dependent gradient estimate:
    \[  \| Du \|_{L^{\infty}(B_{1/2})} \leq C_1(n) \exp \left( C_{2}(n) \|u\|_{L^{\infty}(B_1)} \right).  \]
    Furthermore, Finn's example \cite{Finn} shows that this linear exponential dependence is optimal and cannot be improved. For the scalar curvature equation (or more generally, the $\sigma_k$-curvature equation), Korevaar's gradient estimate has a quadratic exponential dependence. Although there is no example showing whether this quadratic exponential  dependence is optimal, we believe that the gradient estimate for the scalar curvature equation cannot be better than the linear exponential  dependence, as the scalar curvature equation is a fully nonlinear equation that is more complicated than the minimal surface equation. 

    To overcome this difficulty, we note that the goal of Bombieri-De Giorgi-Miranda and Korevaar in proving gradient estimates is to show the boundedness of the solution's gradient. For our purpose of proving the Alexandrov regularity, we only need the linear dependence at small scales. Therefore, under the assumption that the gradient of the solution is bounded on $B_1$, we can improve the gradient bound to the form \eqref{eq: 5.12} at smaller scales.

    Also, for our subsequent applications involving rescaling, we state our proposition for equations with an arbitrary constant right-hand side.
    \begin{proposition}\label{PROP: linear dependent grad est}
        Let $\Sigma=(x,u(x))$ be a smooth $2$-convex graph over $B_1\subset\bb{R}^n$ satisfying 
        \[ \sigma_{2}(\kappa)=f_0 \quad \text{on}\quad  B_1, \]
        for some positive constant $f_0$. Suppose that $u$ has an a priori gradient bound $\|Du\|_{L^{\infty}(B_1)}=\Gamma$, then for any ball $B_{r}(x)\subset B_1$, we have the refined gradient estimate:
        \begin{equation}\label{eq: grad est 5.13}
            \sup_{B_{\frac{r}{2}(x)}}|Du| \leq \dfrac{C(n,\Gamma)}{r} \left(1+\sqrt{f_0} \right)  \underset{B_{r}(x)}{\mathrm{osc}} u.
        \end{equation}
    \end{proposition}

    \begin{proof}
        By translation, we may assume that $x=0$ without loss of generality. Consider the linear rescaling:
        \[ \widetilde{\Sigma}:= \dfrac{\Sigma}{r}= \{ \left(x, v(x)\right) : x\in B_{1/r} \}, \quad \text{where}\ 0<r<1 \ \text{and}\ v(x)=\dfrac{u(rx)}{r}. \]
        Under this rescaling, the scalar curvature of $\widetilde{\Sigma}$ satisfies
        \begin{equation}
            \sigma_2(\kappa(v)) = r^2 f_0. 
        \end{equation}
        It suffices to estimate $\sup_{B_{1/2}}|Dv|$ in terms of $M=\mathrm{osc}_{B_1}v$. Hereafter, all subsequent computations are carried out on $\widetilde{\Sigma}$. We adopt the notations from Section \ref{Section: Preliminaries}, and denote the unit normal on $\widetilde{\Sigma}$ by $\nu=\frac{1}{W}(Dv,-1)$ for $W=\sqrt{1+|Dv|^2}$, and the position vector by $X=(x,v(x))$. We also choose an orthonormal frame $\{e_1, \cdots, e_n, \nu\}$ on $\widetilde{\Sigma}$. Note that $v$ can be viewed as a function on $\widetilde{\Sigma}$ by $v=\innerproduct{X, E_{n+1}}$, then we have
        \[ v_i=\nabla_{e_i} v = \innerproduct{e_i, E_{n+1}}\quad \text{and} \quad  v_{ij}=\nabla^2v(e_i,e_j)=-h_{ij}\innerproduct{\nu, E_{n+1}}=\dfrac{h_{ij}}{W}.\]

       By replacing $v$ with $v+M-\inf_{B_1}v$, we may assume that $M\leq v\leq 2M$ on $B_1$. Consider the following test function on $B_1$:
       \[ P= \eta |Dv| + \dfrac{A}{2}v^2, \]
       where $\eta= 1-|x|^2= 1-|X|^2+v^2$ is a cutoff, and $A$ is a constant to be fixed later. This test function is similar to the one used by Trudinger \cite{Trudinger-CPDE} in his proof of gradient estimate for $\sigma_k$-Hessian equations. It was later used by Warren and Yuan \cite{Warren-Yuan-AJM-10}, as well as Bhattacharya, Mooney and Shankar \cite{Bhattacharya-Mooney-Shankar}, to establish gradient estimates for the special Lagrangian equation and the Lagrangian mean curvature equation, respectively.

       In order to do computations on the hypersurface $\widetilde{\Sigma}$, note that $W=\sqrt{1+|Dv|^2}$,  we can write 
       \[ \alpha:= |Dv|=\sqrt{W^2-1}.  \]
       \begin{claim}[Jacobi inequality for $\alpha$]\label{CLAIM: claim3} 
           There holds $\Delta_{F}\alpha:= F^{ij}\alpha_{ij}\ge 0 $, where $F^{ij}=H\delta_{ij}-h_{ij}$ is the linearized operator.
       \end{claim}
       The proof of this claim is postponed to the end of the main argument. Throughout this proof, $C=C(n,\Gamma)$ denotes the universal constant which may change line by line.
       
        Suppose that $P$ attains its interior maximum at $x_0\in B_1$. Then at $X_0=(x_0, v(x_0))$, we have
        \begin{equation}\label{eq: DP=0 5.15}
            0= P_i = \eta_i \alpha + \eta \alpha_i + Avv_i
        \end{equation}
        and
        \begin{equation}\label{eq: negative ineq 5.16}
            0\ge F^{ij}P_{ij}= \underbrace{\left(F^{ij}\eta_{ij}\right)\alpha}_{:= I}  + \underbrace{2F^{ij}\eta_i \alpha_j}_{:=II} + \underbrace{\eta \left( F^{ij}\alpha_{ij} \right)}_{\ge 0} + \underbrace{AF^{ij}v_iv_j}_{:= III} + \underbrace{AvF^{ij}v_{ij}}_{\ge 0}.
        \end{equation}

        We may rotate the tangent frame $\{e_1,\cdots, e_n\}$ appropriately, such that $\nabla v = v_1 e_1$. That is
        \[  v_i = \innerproduct{e_i, E_{n+1}} =0 ,\quad \text{for}\ i\ge 2, \]
        and 
        \[  v_1= \innerproduct{e_1, E_{n+1}} = \sqrt{1-\innerproduct{\nu, E_{n+1}}^2}= \sqrt{1-\dfrac{1}{W^2}}=\dfrac{\alpha}{W}.   \]
        
        \emph{Estimate on I.} By a computation analogous to that in Section \ref{Section: Doubling ineq}, we have from \eqref{eq: est on rho_ii} that 
         \[ I\ge -C(H+r^2 f_0)\alpha. \]

        \emph{Estimate on II.} 
        From the $DP=0$ equation \eqref{eq: DP=0 5.15}, we have 
        \[ \alpha_j=-\dfrac{1}{\eta} \left( \eta_j\alpha + Avv_j \right) . \]
        Substituting this into the expression for $II$, we obtain
        \begin{align*}
            II = -\dfrac{2\alpha}{\eta} F^{ij}\eta_i \eta_j - \dfrac{2Av}{\eta} F^{ij}\eta_i v_j \ge -\dfrac{CH}{\eta}\alpha - \dfrac{CAMH}{\eta W}\alpha,
        \end{align*}
        where we have used the bound $|\nabla\eta|\leq C$.
        
        \emph{Estimate on III.} Now $III= A F^{11}v_1^2 = \frac{A}{W^2} F^{11}\alpha^2$. We need to estimate the lower bound of $F^{11}=H-h_{11}$. From the $DP=0$ equation \eqref{eq: DP=0 5.15} again, we have
        \[ \dfrac{WW_1}{\alpha} = \alpha_1 = -\dfrac{1}{\eta} \left( \eta_1 + \dfrac{Av}{W} \right)\alpha. \]
        Note that $W_1=W^2\innerproduct{e_k,E_{n+1}}h_{1k} =\alpha W h_{11} $, and substituting this into the above identity, we deduce that
        \[ h_{11} = -\dfrac{1}{\eta W^2} \left( \eta_1 + \dfrac{Av}{W} \right)\alpha. \]
        We fix $A=C/M$ for a sufficiently large constant $C=C(n,\Gamma)$ such that $A\ge  |\nabla\eta| W/M. $ Then we have $h_{11}\leq0$, and thus $F^{11}=H-h_{11}\ge H.$ This gives the lower bound
        \[ III\ge \dfrac{AH}{W^2}\alpha^2. \]

        Combining the above estimates for $I,II$ and $III$, \eqref{eq: negative ineq 5.16} implies that
        \begin{align*}
            0\ge -C(H+r^2f_0)\alpha - \dfrac{CAMH}{\eta W}\alpha + \dfrac{AH}{W^2}\alpha^2.
        \end{align*}
        It follows that
        \[ \eta\alpha\leq \dfrac{CW^2}{A} \left(1+ \dfrac{r^2f_0}{H}\right) \leq \dfrac{C}{A} \left(1+ r\sqrt{f_0} \right)\leq CM \left(1+\sqrt{f_0} \right) , \]
        where we used $H=\sqrt{2r^2 f_0+|\kappa|^2} \ge r \sqrt{f_0}$ and $r\leq 1$. Finally,  we conclude that
        \begin{align*}
            \sup_{B_{1/2}} |Dv| \leq C\sup_{B_1}P \leq CP(x_0) \leq CM \left(1+\sqrt{f_0} \right) + CAM^2 \leq CM \left(1+\sqrt{f_0}\right).
        \end{align*}
        Rescaling back to $u$, we complete the proof.
    \end{proof}

    \noindent\emph{Proof of Claim \ref{CLAIM: claim3}.} We prove this Jacobi inequality pointwise. For any fixed point $p$, by choosing $\{e_1,\cdots, e_n\}$ appropriately, we may assume that $(h_{ij})$ is diagonal at $p$ with $h_{ii}=\kappa_i$ and $\kappa_1\ge \cdots \ge \kappa_n$. Consequently, $\left(F^{ij}\right)$ is also diagonal with $F^{ii}=H-\kappa_i$.  From the similar computation as in Section \ref{Section: Doubling ineq}, we have from \eqref{eq: deravaties of W} that
       \begin{equation}
            W_i = W^2 \innerproduct{e_i, E_{n+1}} \kappa_i \quad \text{and}\quad \sum_{i} F^{ii}W_{ii} = \dfrac{2}{W} \sum_{i} F^{ii}W_{i}^2 + W\sum_{i} F^{ii}\kappa_i^2.
        \end{equation}
        Therefore, 
        \begin{align*}
            F^{ij}\alpha_{ij}&= \dfrac{W}{\alpha}F^{ij}W_{ij} - \dfrac{1}{\alpha^3} F^{ij}W_{i}W_{j}\\
            &= \dfrac{W}{\alpha} \left( \dfrac{2}{W}\sum_{i}F^{ii}W_{i}^2 + W\sum_{i}F^{ii}\kappa_i^2 \right)- \dfrac{1}{\alpha^3} \sum_i F^{ii}W_{i}^{2}\\
            &\ge \dfrac{W^2}{\alpha} \sum_{i} F^{ii}\kappa_i^2 - \dfrac{W^4}{\alpha^3}\sum_{i} F^{ii}\innerproduct{e_i, E_{n+1}}^2 \kappa_i^2\\
            &= \dfrac{W^2}{\alpha}\sum_i F^{ii}\kappa_{i}^2 \left( 1- \dfrac{W^2}{\alpha^2} \innerproduct{e_i, E_{n+1} }^2 \right)\ge 0.
        \end{align*}
        In the last inequality, we used the fact that
        \[ \innerproduct{e_i, E_{n+1}}^2 \leq |E_{n+1}^{\top}| = 1-\innerproduct{\nu, E_{n+1}}^2 = 1-\dfrac{1}{W^2} = \dfrac{\alpha^2}{W^2}. \]
        \hfill\qed

    Next, we need to establish a gradient estimate of the form \eqref{eq: grad est 5.13} for $u$ subtracting its linear part. It should be noted that, after subtracting a linear function $l$ from $u$, $g=u-l$ no longer satisfies the scalar curvature equation. To overcome this, we rotate the graph $\Sigma=(x,u(x))$ in $\bb{R}^{n+1}$ such that $Du$ vanishes at a given point, thus eliminating the linear function to be subtracted.

    We also note that after rotation, $\Sigma$ is not necessarily still globally graphical; it is only locally graphical. Therefore, we need to introduce the modulus of continuity of $Du$ to estimate the size of the domain where $\Sigma$ remains graphical.

    \begin{proposition}\label{PROP: grad est for g=u-l}
         Let $\Sigma=(x,u(x))$ be a smooth $2$-convex graph over $B_1\subset\bb{R}^n$ satisfying 
        \[ \sigma_{2}(\kappa)=1\quad \text{on}\quad  B_1. \]
        Suppose that $u$ has an a priori gradient bound $\|Du\|_{L^{\infty}(B_1)}=\Gamma$. For any fixed $x_0\in B_1$, let $g(x) = u(x) - u(x_0) - Du(x_0)\cdot (x-x_0) $. Then we have the gradient estimate:
        \begin{equation}
            \sup_{B_{r}(y)} |Dg| \leq  \dfrac{C_1}{r} \underset{ B_{C_2 r}(y) }{\mathrm{osc}} g,\quad \text{ for any small ball}\ B_{r}(y)\subset B_{r_0},
        \end{equation}
        where $C_1 , C_2$ are universal constants depending on $n$ and $\Gamma$, and $r_0>0$ is a small constant depending on $n,\Gamma$ and the modulus of continuity of $Du$.
    \end{proposition}
    \begin{proof}
        By translation and subtracting a constant from $u$, we may assume without loss of generality that $x_0=0$ and $u(0)=0$. Let $\{E_1, \cdots, E_n, E_{n+1}\} $ denote the standard orthonormal coordinate of $\bb{R}^{n+1}$. If $|Du(0)|=0$, there is nothing to prove. We may therefore assume that $|Du(0)|>0$. By rotating the horizontal basis $\{ E_1, \cdots, E_n \}$ appropriately, we may further assume that $Du(0)= |Du(0)| E_n$ with $ |Du(0)| = \tan\theta_0$  for some $\theta_0\in \left( 0, \frac{\pi}{2} \right)$.
        Throughout this proof, $C=C(n,\Gamma)$ denotes the universal constant which may change line by line. 
        
        We now perform a counterclockwise rotation of the coordinate plane $\{E_n , E_{n+1}\}$  by angle $\theta_0$, such that the tangent plane of $\Sigma$ at the origin is horizontal in the new coordinate. Precisely, the new orthonormal coordinate is given by:
        \begin{align*}
            \begin{cases}
                \widetilde{E}_i = E_i \qquad \text{for}\ i=1,\cdots, n-1;\\
                \widetilde{E}_n = \cos\theta_0 E_n + \sin \theta_0 E_{n+1};\\
                \widetilde{E}_{n+1}= -\sin\theta_0 E_n + \cos\theta_0 E_{n+1}.
            \end{cases}
        \end{align*}
        where $(\cos\theta_0, \sin\theta_0)=\left( \frac{1}{W(0)}, \frac{|Du(0)|}{W(0)} \right)$.
        \begin{claim}
            There exists a small constant $r_0>0$ depending on $\omega_{Du}$, i.e., the modulus of continuity of $Du$, such that $\Sigma\cap (B_{r_0}\times \bb{R})$ remains graphical under this new coordinate.
        \end{claim}
        \noindent\emph{Proof of Claim.} By the continuity of $Du$, there exists $r_0>0$ depending on $\omega_{Du}$, such that $\nu\cdot \nu(0)\ge 1/2$ in $B_{r_0}$. Under the new rotated coordinate, the hypersurface $\Sigma\cap (B_{r_0}\times \bb{R})$ admits the parametrization:
        \begin{align*}
            \begin{cases}
                \widetilde{x}_i = x_i \qquad \text{for}\ i=1,\cdots, n-1;\\
                \widetilde{x}_n = \cos\theta_0 x_n + \sin\theta_0 u(x);\\
                \widetilde{x}_{n+1} = -\sin\theta_0 x_n + \cos \theta_0 u(x).
            \end{cases}
        \end{align*}
        To prove the graphicality of $\Sigma\cap (B_{r_0}\times \bb{R})$, it suffices to rule out the existence of two distinct points on $\Sigma\cap (B_{r_0}\times \bb{R})$ with identical horizontal coordinates, i.e., pairs $(\widetilde{x}, \widetilde{x}_{n+1}), (\widetilde{x}', \widetilde{x}_{n+1}')\in \Sigma\cap (B_{r_0}\times \bb{R})$ with $\widetilde{x}= \widetilde{x}'$, but $\widetilde{x}_{n+1}\neq \widetilde{x}_{n+1}'$. 
        
        Suppose for contradiction that such pairs exist, corresponding to points $(x,u(x))$ and $(x', u(x'))$ under the original coordinate, respectively. From $\widetilde{x}=\widetilde{x}'$ and the mean value theorem, we deduce that
        \begin{equation}\label{eq: compute the dist change 5.19}
            \begin{split}
                 0= |\widetilde{x}-\widetilde{x}'|&=|\cos\theta_0 (x_n-x_n') + \sin\theta_0 (u(x)-u(x')) |\\
                &= |\cos\theta_0 (x_n-x_n') + \sin\theta_0 Du(x^*)\cdot (x-x')  |\\
                &= |\cos\theta_0 + \sin\theta_0 u_n(x^*)||x_n-x_n'|.
            \end{split}
        \end{equation}
        Using the definition of $\theta_0$, we compute:
        \begin{align*}
            \cos\theta_0 + \sin\theta_0 u_n(x^*)  
            &= \dfrac{1}{W(0)} (1+ Du(0)\cdot Du(x^*))\\
            &= W(x^*) (\nu(0)\cdot \nu(x^*)) \\
            &\ge \dfrac{1}{2}.
        \end{align*}
        It follows that $x=x'$, this gives
        \[ \widetilde{x}_{n+1}= -\sin\theta_0 x_n + \cos\theta_0 u(x)= -\sin\theta_0 x_n' + \cos\theta_0 u(x')= \widetilde{x}_{n+1}', \]
        which contradicts our assumption of distinct points. The graphicality of $\Sigma\cap (B_{r_0}\times \bb{R})$ follows. \hfill\#

        We now denote the new graph function by $\widetilde{u}(\widetilde{x})$, which satisfies
        \[ \widetilde{u}(\widetilde{x})= -\sin\theta_0 x_n + \cos\theta_0 u(x) = \dfrac{1}{W(0)} (u(x)- |Du(0)|x_n)= \dfrac{1}{W(0)} g(x). \]
        Next, we compute the relation between derivatives of $u$ and $\widetilde{u}$. By the chain rule, for $i=1,\cdots, n-1$,
        \begin{align*}
            \widetilde{u}_i = \dfrac{\partial \widetilde{u}}{\partial \widetilde{x}_i} = \sum_{k=1}^{n} \dfrac{\partial \widetilde{u}}{\partial x_k}\dfrac{\partial x_k}{\partial \widetilde{x}_i}= \dfrac{\partial \widetilde{u}}{\partial x_i} + \dfrac{\partial \widetilde{u}}{\partial x_n}\dfrac{\partial x_n}{\partial \widetilde{x}_i}= \cos\theta_0 u_i + (\cos\theta_0 u_n -\sin\theta_0)(-\sin\theta_0\widetilde{u}_i),
        \end{align*}
        and
        \begin{align*}
            \widetilde{u}_n = \dfrac{\partial \widetilde{u}}{\partial \widetilde{x}_n} = \sum_{k=1}^{n} \dfrac{\partial \widetilde{u}}{\partial x_k}\dfrac{\partial x_k}{\partial \widetilde{x}_n}=  \dfrac{\partial \widetilde{u}}{\partial x_n}\dfrac{\partial x_n}{\partial \widetilde{x}_n} = (-\sin\theta_0 + \cos\theta_0 u_n) (\cos\theta_0 - \sin\theta_0 \widetilde{u}_n).
        \end{align*}
        Rearranging these equations to solve $\widetilde{u}_i$ and $\widetilde{u}_n$, we obtain
        \begin{equation}\label{eq: relation bet derivatives}
            \widetilde{u}_i = \dfrac{u_i}{\cos\theta_0 + \sin\theta_0 u_n} \quad\text{and}\quad  \widetilde{u}_n = \dfrac{\cos\theta_0 u_n -\sin\theta_0}{\cos\theta_0 + \sin\theta_0 u_n}.
        \end{equation}
        Note that $\cos\theta_0 + \sin\theta_0 u_n = W(\nu\cdot \nu(0)) \approx 1$, and $\cos\theta_0 u_n -\sin\theta_0 = \frac{1}{W(0)} \left( u_n -|Du(0)| \right)$. We deduce that
        \[ |D\widetilde{u}(\widetilde{x})| \approx |Du(x)- Du(0)|= |Dg(x)|\quad \text{for}\ x\in B_{r_0}. \]
        In particular, $|D\widetilde{u}(\widetilde{x})|\leq C$ on $\widetilde{x}(B_{r_0})$. 

        From the computation \eqref{eq: compute the dist change 5.19}, we know that $C^{-1}|x-x'|\leq |\widetilde{x}-\widetilde{x}'| \leq C |x-x'| $ for all $x,x'\in B_{r_0}$, which means that the coordinate change map $x\mapsto \widetilde{x}$ is a bi-Lipschitz homeomorphism from $B_{r_0}$ to $\widetilde{x}(B_{r_0})$ with universal Lipschitz constants. Thus, for any ball $B_{r}(y)\subset B_{r_0}$,  we have the inclusion
        \[ \widetilde{B}_{r/C}(\widetilde{y}) \subset \widetilde{x}(B_r(y))\subset \widetilde{B}_{Cr}(\widetilde{y}). \]
        where $\widetilde{y}=\widetilde{x}(y)$ denotes the image of $y$ under the coordinate change, and $\widetilde{B}_{r}(\widetilde{y})$ denotes the ball in the $\widetilde{x}$-coordinate.
        
        Note that $\widetilde{u}$ still satisfies the scalar curvature equation on $\widetilde{B}_{r_0/C}\subset \widetilde{x}(B_{r_0})$. For any small ball $B_{r}(y)\subset B_{\frac{r_0}{2C^2}}$, we have $|\widetilde{y}|\leq C|y|\leq \frac{r_0}{2C}$ by the bi-Lipschitz bound. Combined with $r<\frac{r_0}{2C^2}$, this gives $\widetilde{x}(B_{r}(y))\subset \widetilde{B}_{Cr}(\widetilde{y})\subset \widetilde{B}_{r_0/C}$. Applying the rescaled version of Proposition \ref{PROP: linear dependent grad est} to $\widetilde{u}$ on $\widetilde{B}_{Cr}(\widetilde{y})$, we obtain
        \[\sup_{\widetilde{B}_{Cr}(\widetilde{y})}|D\widetilde{u}|\leq \dfrac{C}{r}\underset{\widetilde{B}_{2Cr}(\widetilde{y})}{\mathrm{osc}}\,\widetilde{u}.\]
        Returning to the original coordinate, we conclude that
        \[ \sup_{B_{r}(y)}|Dg|\leq \dfrac{C}{r}\underset{B_{2C^2r}(y)}{\mathrm{osc}}\, g,\quad \text{for all small balls}\ B_{r}(y)\subset B_{\frac{r_0}{2C^2}} . \]
    \end{proof}

    \begin{corollary}\label{COR: Lip est}
        Under the hypotheses of Proposition \ref{PROP: grad est for g=u-l}, the gradient estimate can be improved to 
        \[ \sup_{ \substack{x,y\in B_{r}(x_0) \\ x\neq y} } d_{x,y}^{n+1} \dfrac{|g(x)-g(y)|}{|x-y|} \leq C \int_{B_{r}(x_0)}|g|, \quad \text{for all small}\ r<r_0, \]
        where $d_x=\mathrm{dist}(x, \partial B_r(x_0))= r-|x-x_0|$, $d_{x,y}=\min\{d_x, d_y\}$,  $C$ is a positive constant depending on $n$ and $\Gamma$, and $r_0>0$ is a small constant depending on $n,\Gamma$ and the modulus of continuity of $Du$.
    \end{corollary}
    \begin{proof}
        By translation, it suffices to assume that $x_0=0$. We first introduce some notations from \cite[p. 585]{Trudinger-Wang}. For a continuous function $u$ on $B_1$ and any $r\in (0,1)$, define the weighted interior norms and semi-norms:
        \[ |u|_{0;r}^{(n)}:= \sup_{x\in B_r} d_x^n |u(x)|, \quad [u]_{0,1; r}^{(n)}= \sup_{\substack{x,y\in B_r\\ x\neq y}}d_{x,y}^{n+1}\dfrac{|u(x)-u(y)|}{|x-y|}. \]
        There is an interpolation inequality \cite[Lemma 2.6]{Trudinger-Wang} between these norms: for any $\varepsilon>0$, we have
        \begin{equation}\label{eq: interpolation}
            |u|_{0;r}^{(n)} \leq \varepsilon [u]_{0,1;r}^{(n)} + C(n)\varepsilon^{-n} \int_{B_r}|u|.
        \end{equation}

        Now fix any $r<r_0$ and distinct points $x,y\in B_r$, and denote $d=d_{x,y}.$ For $t\in[0,1]$, set $z_t= tx+(1-t)y$, then $B_{d/2}(z_t)\subset B_{r-d/2}$. By the fundamental theorem of calculus, we have
        \[ g(x)-g(y)= \int_{0}^{1}\dfrac{\mathrm{d}}{\mathrm{d}t} g(z_t) \,\mathrm{d}t = \int_{0}^{1} Dg(z_t)\cdot (x-y)\,\mathrm{d}t. \]
        Applying Proposition \ref{PROP: grad est for g=u-l} on $B_{\frac{d}{2C_2}}(z_t)$ yields
        \[ |Dg(z_t)|\leq \sup_{B_{\frac{d}{2C_2}}(z_t)}|Dg| \leq \dfrac{C_1}{d}\sup_{B_{d/2}(z_t)}|g| \leq \dfrac{C_1}{d} \sup_{B_{r-d/2}}|g|. \]
        where $C_1, C_2$ are universal constants consistent with Proposition \ref{PROP: grad est for g=u-l}.
        Therefore, 
        \[ d^{n+1}\dfrac{|g(x)-g(y)|}{|x-y|}\leq C_{1} d^n \sup_{B_{r-d/2}}|g| \leq C_1 |g|_{0;r}^{(n)}. \]
        Taking the supremum over $x,y$ and using the interpolation inequality \eqref{eq: interpolation}, we have for any $\varepsilon>0$ that
        \[ [g]_{0,1;r}^{(n)}\leq C_1|g|_{0;r}^{(n)} \leq C_1 \left( \varepsilon [g]_{0,1;r}^{(n)} + C(n)\int_{B_r}|g| \right). \]
        Choosing $\varepsilon=\frac{1}{2C_1}$, we complete the proof.
    \end{proof}

\section{Proof of curvature estimates}\label{Section: Proof of Curvature Estimates}

\noindent\emph{Proof of Theorem \ref{THM: curvature est}.} \emph{Step 1.} If the curvature estimate \eqref{eq: curvature est} were false, there would exist a sequence of graphs $\Sigma_k = (x, u_k(x))$ such that:

$(i)$ $\Sigma_k$ is a smooth $2$-convex graph satisfying the scalar curvature equation \eqref{eq: scalar curvature eq} on $B_1\subset \bb{R}^4$;

$(ii)$ $\|u_k\|_{C^1(B_1)}\leq A$ and each $Du_k$ has a uniform modulus of continuity $\omega$;

$(iii)$ but the curvature $|\kappa(u_k)(0)|\to\infty$ as $k\to\infty$. 

By $(ii)$, applying the Arzela-Ascoli theorem, up to a subsequence, $u_k$ converges to a $C^1$ function $u$ in $C^{1}_{\mathrm{loc}}(B_1)$. By the closedness of viscosity solutions (see \cite{CCbook}), $u$ is a viscosity solution to \eqref{eq: scalar curvature eq}. Since $u$ is $C^1$, hypothesis $(\mathrm{H}1)$ of Proposition \ref{PROP: general Alexandrov} holds trivially.  By Proposition \ref{PROP: Hessian as Radon measures},  hypothesis $(\mathrm{H}2)$ is also satisfied. By Corollary \ref{COR: Lip est}, the Lipschitz estimate \eqref{eq: H3} holds uniformly for all $u_k$, passing to the limit, $(\mathrm{H}3)$ holds for $u$. Therefore, we conclude that $u$ is twice differentiable almost everywhere in $B_1$.

\vspace{0.3cm}
\emph{Step 2. Universally small gradient bound.} Up to subtracting a constant from $u$ and rotating the graph of $u$, we may assume that 
\[ u(0)= Du(0)= 0 \qquad \text{and}\qquad \sup_{B_{r}}|u|\leq \sigma(r) \quad \text{for all}\ r\leq r_0, \]
where $r_0$ is a small constant depending only on $u$, and $\sigma(r)$ is a modulus of order $o(r)$ as $r\to 0$.

Applying our linearly dependent gradient estimate (Proposition \ref{PROP: linear dependent grad est}) to $u_k$, we have for $r<r_0$ that
\[ \sup_{B_{\frac{r}{2}}}|Du_k| \leq \dfrac{C(4,A)}{r}\sup_{B_r}|u_k|\leq \dfrac{C}{r} \left( \sup_{B_r}|u-u_k|+\sigma(r) \right). \]
We fix a universally small $r_1>0$ such that $\sigma(r_1)/r_1 < \mu/2$, where $\mu$ is the dimensional constant appearing in the small gradient condition of Corollary \ref{COR: rescaled doubling}. For sufficiently large $k$, we can ensure that 
\[ \sup_{B_{\frac{r_1}{2}}}|Du_k|\leq \mu. \]
Thus, by Corollary \ref{COR: rescaled doubling}, the doubling inequality holds for $u_k$ uniformly in $B_{r_1/6}$. That is, for any $y\in B_{r_1/ 18}$ and $0<\rho < r_1/9$, we have
\begin{equation}\label{eq: 6.1}
    \sup_{B_{\frac{r_1}{6}}} H(u_k) \leq C(u,A, r_1, \rho) \sup_{B_{\rho}(y)} H(u_k).
\end{equation}
We emphasize that the constant $C$ is independent of $k$.

\vspace{0.3cm}
\emph{Step 3. Applying Savin's small perturbation theorem.} Define the fully nonlinear elliptic operator $G(M, p)= \sigma_{2}(g^{-1}h)-1$, where 
\[g^{-1}= (g^{ij})= \left( \delta_{ij} - \frac{p_i p_j}{1+|p|^2} \right) \quad  \text{and}\quad h=(h_{ij})= \left(\frac{M_{ij}}{\sqrt{1+|p|^2}}\right).\]
It is clear that $G(D^2 u, Du)= \sigma_{2}(\kappa)-1=0$. Moreover, the operator $G$ is uniformly elliptic in a neighborhood of $(0,0)\in \mathcal{S}^{4\times 4}\times \bb{R}^4$, and $G\in C^2$ with $|D^2 G|\leq C(4)$.

We fix a point $y\in B_{r_1/18}$ where $u$ is twice differentiable, and let $Q(x)$ be the quadratic Taylor polynomial of $u$ at $y$, so that $|u(x)-Q(x)|=o(|x-y|^2)$ near $y$.
By definition, $Q$ satisfies $G(D^2Q, DQ(y))=0$.

Let $v_k=u_k -Q$.  For small $r>0$, consider the rescaled function 
\[ \widetilde{v}_k(x)= \dfrac{1}{r^2} v_k (rx+y)\quad \text{for}\ x\in B_1. \]
Then $\widetilde{v}_k$ satisfies the equation
\begin{align*}
     G(D^2\widetilde{v}_k(x) + D^2Q, rD\widetilde{v}_k(x) + DQ(rx+y)) =0 .
\end{align*}
We now apply the generalized Savin's small perturbation theorem \cite{Savin, Fan} to $\widetilde{v}_k$. Consider the operator 
\[ \widetilde{G}(M,p,x)= G(M+D^2Q, rp + DQ(rx+y))- G(D^2Q, DQ(rx+y)). \]
By construction, $\widetilde{G}(0,0,x)\equiv0$ and $\widetilde{G}$ satisfies all hypotheses of generalized Savin's small perturbation theorem \cite[Theorem 1.7]{Fan}. Let $\delta, C$ be the universal constants from \cite[Theorem 1.7]{Fan}. Now $\widetilde{v}_k$ solves
\[ \widetilde{G}(D^2\widetilde{v}_k, D\widetilde{v}_k, x)=  -G(D^2Q, DQ(rx+y)):= f(x). \]
We verify the conditions of \cite[Theorem 1.7]{Fan}. Fix any $\alpha\in (0,1)$, we have
\begin{align*}
    \|\widetilde{v}_k\|_{L^{\infty}(B_1)} &\leq \dfrac{\|u_k-u\|_{L^{\infty}(B_{r}(y))}}{r^2} + \dfrac{\| u-Q \|_{L^{\infty}(B_{r}(y))}}{r^2}\\
    &\leq  \dfrac{\|u_k-u\|_{L^{\infty}(B_{r}(y))}}{r^2}  + \dfrac{o(r^2)}{r^2},
\end{align*}
as well as
\begin{equation*}
    |\widetilde{G}(M,p,x)-\widetilde{G}(M,p,x')|\leq C(4,Q) r|x-x'|, \quad \|f\|_{C^{0,\alpha}(B_1)} \leq C(4,Q) r^{\alpha}.
\end{equation*}
We first fix $\rho$ universally small enough, such that $o(\rho^2)/\rho^2 + C(4,Q)\rho^{\alpha} <\delta/2$. Then for sufficiently large $k$, we can ensure that
\[ |\widetilde{G}(M,p,x)-\widetilde{G}(M,p,x')|\leq \delta|x-x'|^{\alpha}, \]
and
\[ \|\widetilde{v}_k\|_{L^{\infty}(B_1)}\leq \delta,\quad \|f\|_{C^{0,\alpha}(B_1)}\leq \delta. \]
Hence, by \cite[Theorem 1.7]{Fan}, we obtain $\|\widetilde{v}_k\|_{C^{2,\alpha}(B_{1/2})}\leq C$, where $C$ is independent of $k$. This yields the uniform curvature bound:
\[  H(u_k) \leq C \ \text{on}\ B_{\rho/2}(y)\ \text{uniformly in}\ k. \]

\vspace{0.3cm}
\emph{Step 4.} Applying the doubling inequality \eqref{eq: 6.1}, we obtain
\[ \sup_{B_{\frac{r_1}{6}}} H(u_k)\leq C, \]
where $C$ is independent of $k$. This contradicts the curvature blow-up assumption $(iii)$ that $|\kappa(u_k)(0)|\to \infty$ as $k\to \infty$. \hfill\qed.

\section{Alexandrov regularity revisited}\label{Section: Alexandrov Regularity Revisited}
The previously established curvature estimates depend on the modulus of continuity of the gradient. Hence, the regularity of $C^1$ viscosity solutions to \eqref{eq: scalar curvature eq} cannot be directly derived from these estimates via the standard approximation argument. Indeed, when solving the Dirichlet problem with smooth approximating boundary data, the resulting smooth approximating solutions do not necessarily admit a uniform modulus of continuity of the gradient. Therefore, the regularity of $C^1$ viscosity solutions to \eqref{eq: scalar curvature eq} remains a nontrivial issue.

In the previous proof of Alexandrov regularity, the modulus of continuity of the gradient was only used to estimate the size of the graphical domain in our rotation argument (see Section \ref{subsec: Linearly dependent gradient estimates}). In this section, we refine the rotation argument to establish the following general Alexandrov regularity for viscosity solutions to \eqref{eq: scalar curvature eq}.

\begin{proposition}\label{PROP: Alexandrov revisited}
    Let $\Sigma=(x,u(x))$ be a $2$-convex graph over $B_1\subset \bb{R}^n$, and let $u$ be a viscosity solution to \eqref{eq: scalar curvature eq} on $B_1$. Then $u$ is twice differentiable almost everywhere in $B_1$.
\end{proposition}

By Remark \ref{RMK: Lip reg for viscosity solu}, hypothesis $(\mathrm{H}1)$ holds. Moreover, we can construct a sequence of smooth approximating solutions, and Korevaar's gradient estimates (Proposition \ref{PROP: Korevaar grad est}) guarantee a uniform gradient bound for these solutions. Thus, hypothesis $(\mathrm{H}2)$ follows from Proposition \ref{PROP: Hessian as Radon measures}. It remains to verify hypothesis $(\mathrm{H}3)$. We first establish a priori linearly dependent gradient estimates analogous to Proposition \ref{PROP: grad est for g=u-l}, which do not depend on the modulus of continuity of $Du$.  

\begin{proposition}\label{PROP: gradient and Lip est revisited}
    Let $\Sigma=(x,u(x))$ be a smooth $2$-convex graph satisfying the scalar curvature equation \eqref{eq: scalar curvature eq} on $B_1\subset \bb{R}^n$. Suppose that $u$ has an a priori gradient bound $\|Du\|_{L^{\infty}(B_1)}=\Gamma$. For any fixed $x_0\in B_1$, let $g(x)= u(x)-u(x_0)- Du(x_0)\cdot (x-x_0)$. Assume that there exists a modulus $\sigma(r)= o(r)$ as $r\to 0$ such that $|g(x)|\leq \sigma(|x|)$.  Then, 
    
    $(i)$ there holds the gradient estimate:
    \begin{equation}
            \sup_{B_{r}(y)} |Dg| \leq  \dfrac{C_1}{r} \underset{ B_{C_2 r}(y) }{\mathrm{osc}} g,\quad \text{ for any small ball}\ B_{r}(y)\subset B_{r_0},
        \end{equation}
    where $C_1 , C_2$ and small $r_0>0$  are universal constants depending on $n, \Gamma$ and the modulus $\sigma$.

    (ii) Moreover, as in the proof of Corollary \ref{COR: Lip est}, we can improve the gradient estimate to
    \begin{equation}\label{eq: Lip estimate}
        \sup_{ \substack{x,y\in B_{r}(x_0) \\ x\neq y} } d_{x,y}^{n+1} \dfrac{|g(x)-g(y)|}{|x-y|} \leq C \int_{B_{r}(x_0)}|g|, \quad \text{for all small}\ r<r_0, 
    \end{equation}
    where $d_x=\mathrm{dist}(x, \partial B_r(x_0))= r-|x-x_0|$, $d_{x,y}=\min\{d_x, d_y\}$,  $C$ and small $r_0>0$ are universal constants depending on $n,\Gamma$ and the modulus $\sigma$.
\end{proposition}
If $|Du(x_0)|=0$, this gradient estimate follows from Proposition \ref{PROP: linear dependent grad est}. In the case of $|Du(x_0)|\neq 0$, we again use the rotation argument to eliminate the linear function to be subtracted. The following lemma is the key step in our proof. It states that the rotation argument can proceed at a uniform scale.

\begin{lemma}\label{LEMMA: rotation revisited}
    Let $\Sigma=(x,u(x))$ be a smooth $2$-convex graph satisfying the scalar curvature equation \eqref{eq: scalar curvature eq} on $B_1\subset \bb{R}^n$. Assume that $\|Du\|_{L^{\infty}(B_1)}=\Gamma$, $u(0)=0$ and $Du(0)=|Du(0)|E_n $ with $|Du(0)|= \tan\theta_0$ for some angle $\theta_0\in \left( 0, \frac{\pi}{2} \right)$. Suppose that there exists a modulus $\sigma(r)= o(r)$ as $r\to 0$ such that
    \[ |u(x)- Du(0)\cdot x| = |u(x)-\tan \theta_0 x_n |\leq \sigma(|x|). \]
    Then,
    
    $(i)$ there exists a small universal constant $r^*= r^*(n,\Gamma, \sigma)>0$ such that $\Sigma\cap \left( B_{r^*}\times\bb{R} \right)$ remains graphical under the rotated coordinates:
    \begin{equation}\label{eq: rotated coordinates}
        \begin{split}
            \begin{cases}
            \widetilde{E}_i = E_i\quad \text{for}\ i=1,\cdots, n-1;\\
            \widetilde{E}_n = \cos\theta_0 E_{n}+ \sin\theta_0 E_{n+1};\\
            \widetilde{E}_{n+1} = -\sin\theta_0 E_n + \cos\theta_0 E_{n+1}.
            \end{cases}
        \end{split}
    \end{equation}

    $(ii)$ Denote the new graph function of $\Sigma$ in the rotated coordinates by $\widetilde{u}(\widetilde{x})$, then $\widetilde{u}$ is defined on $\widetilde{x}(B_{r^*})\supset \widetilde{B}_{r^*/2}$.
\end{lemma}
\begin{proof}
    For any $0<r<1$, define
    \[  \Theta_{r}= \sup_{B_r} \arctan u_n \quad \text{and}\quad \theta_r=\inf_{B_r} \arctan u_n.    \]
    It suffices to show that there exists a universal constant $r^*= r^*(n,\Gamma, \sigma)$ such that $\theta_{r^*} > -\frac{\pi}{2} + \theta_0$. This implies that the downward rotation by angle $\theta_0$ is valid at the scale of $B_{r^*}$.

    Our idea is to first rotate $\Sigma$ by the largest possible angle at a larger scale. After this rotation, although the slope of $\Sigma$ may tend to $-\infty$ in the new coordinates, Korevaar's gradient estimate (Proposition \ref{PROP: Korevaar grad est}) guarantees that its slope remains bounded at a smaller interior scale. Hence, returning to the original coordinates, we obtain an improvement on the lower bound of $\theta_r$ at this smaller scale. Finally, by iterating this improvement finitely many times, we obtain the desired $r^*$. 

    \vspace{0.3cm}
    \emph{Step 1. First rotation.} We start at the scale of $B_1$. We assume that $\theta_1 \leq -\frac{\pi}{2} +\theta_0 $; otherwise, we are done. For any $\beta\in \left( 0, \frac{\pi}{2} + \theta_1 \right)$, we rotate the subspace $\{E_{n}, E_{n+1}\}$ counterclockwise by angle $\beta$. Namely, introducing the new coordinates:
    \[ \begin{cases}
            \widetilde{E}_i = E_i\quad \text{for}\ i=1,\cdots, n-1;\\
            \widetilde{E}_n = \cos\beta E{n} + \sin\beta E_{n+1};\\
            \widetilde{E}_{n+1} = -\sin\beta E_n + \cos\beta E_{n+1}.
    \end{cases} \]
    One can check that $\Sigma$ remains graphical under these new coordinates. Denoting the new graph function by $\widetilde{u}$, we next estimate the size of the domain of $\widetilde{u}$.

    \begin{claim}\label{CLAIM: claim 5}
        For any $0<r<1$, we have $\widetilde{x}(B_r)\subset \widetilde{B}_{c_0 r}$ with $c_0= \sqrt{2(1+\Gamma^2)}$.
    \end{claim}
    \noindent\emph{Proof of Claim.} This claim follows from a straightforward computation:
    \[ |\widetilde{x}|^2\leq \sum_{i=1}^{n-1}x_i^2 + (\cos\beta x_n + \sin\beta u(x))^2 \leq 2(|x|^2+ |u(x)|^2)\leq 2(1+\Gamma^2)|x|^2, \]
    where we used the fact that $|u(x)|\leq \Gamma |x|$, since $u(0)=0$ and $\|Du\|_{L^{\infty}(B_1)}=\Gamma$. \hfill\qed

    \begin{claim}\label{CLAIM: claim 6}
        There exists a small constant $\rho_0\in (0,1)$ depending only on $\sigma$ such that for all $r< \rho_0$, we have $\widetilde{B}_{r/2}\subset \widetilde{x}(B_{r})$.
    \end{claim}
    \noindent\emph{Proof of Claim.} For $|x|=r$, we have
    \begin{align*}
        |\widetilde{x}|^2 &= \sum_{i=1}^{n-1} x_i^2 + (\cos\beta x_n +\sin\beta u(x))^2\\
        &= \sum_{i=1}^{n-1} x_i^2 + (\cos\beta x_n +\sin\beta \tan\theta_0 x_n +\sin\beta (u(x)-\tan\theta_0 x_n))^2.
    \end{align*}
    Recalling the definition of the modulus $\sigma$, we have $|u(x)-\tan\theta_0 x_n|\leq \sigma(|x|)$. Using this, we obtain
    \[ |\widetilde{x}|^2 \ge \sum_{i=1}^{n-1} x_i^2 + \dfrac{1}{2} (\cos\beta + \sin\beta \tan \theta_0)x_n^2 - \sigma(r)^2. \]
    Note that $\beta < \theta_0$, we have $0< \theta_0-\beta < \theta_0< \frac{\pi}{2}$, and hence
    \[ \cos\beta + \sin\beta \tan\theta_0 = \dfrac{\cos\beta \cos\theta_0 + \sin\beta \sin\theta_0}{\cos\theta_0} = \dfrac{\cos (\theta_0 -\beta)}{\cos\theta_0} >1. \]
    Therefore,
    \[  |\widetilde{x}|^2 \ge \dfrac{r^2}{2} - \sigma(r)^2 = \dfrac{r^2}{4} \left[ 2- 4\left( \dfrac{\sigma(r)}{r} \right)^2 \right] \ge \dfrac{r^2}{4},  \]
    provided $r\leq \rho_0$ for some $\rho_0=\rho_0(\sigma)$. \hfill\qed
    
    Combining the above two claims, we have
    \begin{equation}\label{eq: relation of domains}
        \widetilde{B}_{r/2}\subset  \widetilde{x}(B_r) \subset \widetilde{B}_{c_0 r}, \quad \text{for any}\ r\leq \rho_0.
    \end{equation}
    
    \vspace{0.3cm}
    \emph{Step 2. Applying Korevaar's gradient estimate for $\widetilde{u}$.} After this rotation, the new graph function $\widetilde{u}$ solves the scalar curvature equation $\sigma_{2}(\kappa(\widetilde{u}))=1$ on $\widetilde{B}_{\rho_0/2} \subset \widetilde{x}(B_1)$. Applying the rescaled version of Korevaar's gradient estimate (Proposition \ref{PROP: Korevaar grad est}) to $\widetilde{u}$ on the ball $\widetilde{B}_{\rho_0/2}$, we obtain
    \begin{align*}
        \sup_{\widetilde{B}_{\rho_0/4}}|D\widetilde{u}| &\leq C(n) \exp \left\{ C(n) \left(1+ \sqrt{\rho_0^2/4}\right) \left(\dfrac{\widetilde{M}}{\rho_0/2} \right)^2 \right\}\\
        &\leq C(n) \exp \left\{ C(n) \left( \dfrac{\widetilde{M}}{\rho_0} \right)^2   \right\},
    \end{align*}
    where $\widetilde{M}= \mathrm{osc}_{\widetilde{B}_{\rho_0/2}}\widetilde{u}$. Next, we estimate $\widetilde{M}$. From the relation \eqref{eq: relation of domains}, we have
    \begin{align*}
        \widetilde{M}&= \sup_{|\widetilde{x}|, |\widetilde{x}'| \leq \rho_0/2 } |\widetilde{u}(\widetilde{x})- \widetilde{u}(\widetilde{x}')| \leq \sup_{|x|,|x'|\leq \rho_0} |\cos\beta (u(x)-u(x')) - \sin\beta (x_n-x_n')|\\
        &\leq \sup_{|x|,|x'|\leq \rho_0} |(\cos\beta \tan\theta_0 - \sin\beta) (x_n-x_n') | + 2\sigma(\rho_0)\\
        &\leq \left( |\cos\beta \tan\theta_0 - \sin\beta| +1  \right) \rho_0,
    \end{align*}
    where we used that $|u(x)-\tan\theta_0 x_n|\leq \sigma(|x|)$ and $\sigma(\rho_0)\leq \rho_0/2$ by our choice of $\rho_0$ in Step 1. Note that
    \[ |\cos\beta \tan\theta_0 - \sin\beta| = \dfrac{|\sin(\theta_0-\beta)|}{\cos\theta_0}\leq \dfrac{1}{\cos\theta_0}\leq \dfrac{1}{\cos(\arctan \Gamma) }. \]
    Therefore, Korevaar's gradient estimate implies that 
    \begin{equation}\label{eq: gradient bound of tilde u}
        \sup_{\widetilde{B}_{\rho_0/4}}|D\widetilde{u}| \leq  C(n,\Gamma):=\tan\gamma ,\quad \text{for some universal}\ \gamma=\gamma(n,\Gamma)\in \left(0,\dfrac{\pi}{2}\right).
    \end{equation}

    \vspace{0.3cm}
    \emph{Step 3. Improvement of $\theta_r$ at smaller scales.} By the chain rule, we compute the relation between $Du(x)$ and $D\widetilde{u}(\widetilde{x})$ as in \eqref{eq: relation bet derivatives}, that is
    \[  \widetilde{u}_i = \dfrac{u_i}{\cos\beta + \sin\beta u_n} \quad\text{and}\quad  \widetilde{u}_n = \dfrac{\cos\beta u_n -\sin\beta}{\cos\beta + \sin\beta u_n}= \tan (\arctan u_n -\beta). \]
    Hence, \eqref{eq: gradient bound of tilde u} implies that
    \[ \inf_{B_{\rho_0/4c_0}} \tan( \arctan u_n-\beta )\ge -\tan\gamma, \quad \Longrightarrow\quad \theta_{\frac{\rho_0}{4c_0}}\ge \beta -\gamma. \]
    Since the above rotation argument holds for any rotation angle $\beta \in \left( 0, \frac{\pi}{2} + \theta_1 \right)$, and $\rho_0, c_0, \gamma$ are independent of $\beta$, we conclude that
    \begin{equation}\label{eq: improvement of theta_r}
        \theta_{\frac{\rho_0}{4c_0}} \ge \theta_1 + \left(\dfrac{\pi}{2}-\gamma\right).
    \end{equation}

    In summary, if at the large scale $B_1$, we have $\theta_1 \leq -\frac{\pi}{2} + \theta_0$, then at the smaller scale, we can improve the lower bound of $\arctan u_n $ by 
    \[ \theta_\eta \ge \theta_1 + \left(\dfrac{\pi}{2}-\gamma \right), \quad \text{where}\ \eta=\eta (\Gamma,\sigma):= \dfrac{\rho_0}{4c_0}. \]
    Since $\eta=\eta(\Gamma,\sigma)$ and $\gamma=\gamma(n,\Gamma)$ are universal constants, we can also start the above rotation argument from any scale $B_r$. Once $\theta_r \leq -\frac{\pi}{2} + \theta_0 $, we can improve it at the smaller scale that
    \[ \theta_{\eta r} \ge \theta_r + \left( \dfrac{\pi}{2} - \gamma \right). \]
    Therefore, repeating this improvement finitely many times, there must exist a universal $r^*=r^*(n,\Gamma, \theta_0, \sigma) \in (0,1)$ such that $\theta_{r^*} > -\frac{\pi}{2} + \theta_0$. 
    
    Note that this $r^*$ may depend on $\theta_0$. Since the angle of each improvement is universal, the number of iterations will not exceed $k_0:= \left[ \frac{\pi}{\pi/2 -\gamma} \right]+1 $. Hence, $r^*$ has a lower bound $r_*:= \eta^{k_0}$ which depends only on $n, \Gamma$ and $\sigma$. Therefore, we can choose $r^*$ smaller so that $r^*= r^*(n,\Gamma,\sigma)$ is independent of $\theta_0$ and  $\theta_{r^*} > -\frac{\pi}{2} + \theta_0$. Consequently, the downward rotation by angle $\theta_0$ is valid at the scale of $B_{r^*}$. 
\end{proof}

\vspace{0.3cm}
\noindent\emph{Proof of Proposition \ref{PROP: gradient and Lip est revisited}.} By translation and subtracting a constant from $u$, we may assume without loss of generality that $x_0=0$ and $u(0)=0$. It suffices to consider the case of $|Du(0)|>0$. By rotating the horizontal basis $\{ E_1, \cdots, E_n \}$ appropriately, we may further assume that $Du(0)= |Du(0)| E_n $ with $|Du(0)|= \tan\theta_0$  for some $\theta_0\in \left( 0, \frac{\pi}{2} \right)$. In this proof, by a universal constant we mean a constant depending only on $n,\Gamma$ and the modulus $\sigma$.

By Lemma \ref{LEMMA: rotation revisited}, there exists a universal constant $r^*\in (0,1)$ such that $\Sigma\cap (B_{r^*}\times\bb{R})$ remains graphical under the rotated coordinates \eqref{eq: rotated coordinates}. Moreover, the new graph function $\widetilde{u}$ is defined on $\widetilde{x}(B_{r^*})\supset \widetilde{B}_{r^*/2}$. Note that
\[ \widetilde{u}(\widetilde{x}) = -\sin\theta_0 x_n + \cos \theta_0 u(x) = \dfrac{1}{W(0)} (u(x)-|Du(0)|x_n) =\dfrac{1}{W(0)} g(x). \]
It follows that $\|\widetilde{u}\|_{L^{\infty}(\widetilde{B}_{r^*/2})} \leq \sigma(r^*/2)$. Since $\widetilde{u}$ solves the scalar curvature equation $\sigma_{2}(\kappa(\widetilde{u})) =1$ on $\widetilde{B}_{r^*/2}$, applying the rescaled version of Korevaar's gradient estimate to $\widetilde{u}$ on $B_{r^*/2}$, we obtain
\begin{align*}
    \|D\widetilde{u}\|_{L^{\infty}(\widetilde{B}_{r^*/4})} &\leq C(n)\exp\left\{ C(n) \left(1 + \dfrac{r^*}{2}\right)  \left( \dfrac{\|\widetilde{u}\|_{L^{\infty}(\widetilde{B}_{r^*/2})}}{r^*/2} \right)^2   \right\} \\
    &\leq C(n) \exp \left\{  C(n) \left(  \dfrac{\sigma(r^*/2)}{r^*/2}\right)^2  \right\} := \widetilde{\Gamma},
\end{align*}
where $\widetilde{\Gamma}$ is a universal constant. Next, we can apply the rescaled version of the linearly dependent gradient estimate (Proposition \ref{PROP: linear dependent grad est}) to $\widetilde{u}$ on $\widetilde{B}_{r^*/4}$ to obtain
    \begin{equation}\label{eq: 7.7}
        \sup_{\widetilde{B}_{r/2}(\widetilde{y})} |D\widetilde{u}| \leq \dfrac{C(n,\widetilde{\Gamma})}{r} \underset{\widetilde{B}_{r}(\widetilde{y})}{\mathrm{osc}} \widetilde{u}, \quad \text{for any small ball}\ \widetilde{B}_{r}(\widetilde{y})\subset \widetilde{B}_{r^*/4}.
    \end{equation}

By similar computations as in the proofs of Claim \ref{CLAIM: claim 5} and Claim \ref{CLAIM: claim 6}, we deduce that the coordinate change map $x\mapsto \widetilde{x}$ is a bi-Lipschitz homeomorphism from $B_{\rho_0}$ to $\widetilde{x}(B_{\rho_0})$ with universal Lipschitz constants, where $\rho_0$ is the universal constant from Claim \ref{CLAIM: claim 6}. Thus, for any ball $B_{r}(y)\subset B_{\rho_0}$,  we have the inclusion
    \[ \widetilde{B}_{r/C}(\widetilde{y}) \subset \widetilde{x}(B_r(y))\subset \widetilde{B}_{Cr}(\widetilde{y}). \]
where $\widetilde{y}=\widetilde{x}(y)$ denotes the image of $y$ under the coordinate change.

For any small ball $B_{r}(y)\subset B_{\frac{r^*}{8C}}$, we have $|\widetilde{y}|\leq C|y|\leq \frac{r^*}{8} $ by the bi-Lipschitz bound. Combined this with $r< \frac{r^*}{8C}$, we have $\widetilde{x}(B_{r}(y)) \subset \widetilde{B}_{Cr}(\widetilde{y}) \subset  \widetilde{B}_{\frac{r^*}{4}}$. From the relation between $D\widetilde{u}(\widetilde{x}) $ and $Du(x)$ as in \eqref{eq: relation bet derivatives}:
\[ \widetilde{u}_i = \dfrac{u_i}{\cos\theta_0 + \sin\theta_0 u_n} \quad\text{and}\quad  \widetilde{u}_n = \dfrac{\cos\theta_0 u_n -\sin\theta_0}{\cos\theta_0 + \sin\theta_0 u_n}= \dfrac{u_{n}-|Du(0)|}{1+ |Du(0)|u_n}. \]
we have
\[ |Dg(x)|=|Du(x)-Du(0)|\leq C(\Gamma) |D\widetilde{u}(\widetilde{x})|. \]
Therefore, the gradient bound \eqref{eq: 7.7} implies that
\begin{equation*}
    \sup_{B_{r}(y)} |Dg| \leq C\sup_{B_{Cr}(\widetilde{y})}|D\widetilde{u}| \leq \dfrac{C}{r}\underset{\widetilde{B}_{2Cr}(\widetilde{y})}{\mathrm{osc}} \widetilde{u} \leq  \dfrac{C}{r} \underset{B_{2C^2r}(y)}{\mathrm{osc}} g. 
\end{equation*}
This completes the proof of part $(i)$. The Lipschitz estimates in part $(ii)$ follow from the same argument as   in the proof of Corollary \ref{COR: Lip est}. \hfill\qed

\vspace{0.3cm}
Now, we are ready to prove the Alexandrov regularity for viscosity solutions to \eqref{eq: scalar curvature eq}.

\vspace{0.3cm}
\noindent\emph{Proof of Proposition \ref{PROP: Alexandrov revisited}.} It suffices to verify $(\mathrm{H}3)$ of Proposition \ref{PROP: general Alexandrov} for $u$. Fix a small constant $\rho_1>0$ such that the curvature condition $\frac{n-2}{n}\sigma_{2}(\kappa_{\partial B_{\rho_1}})>1$ from \cite[Theorem 4.1]{Ivochkina} holds. We solve the Dirichlet problem with smooth approximating boundary data to obtain a sequence of smooth solutions $\{u_k\}$ on $B_{\rho_1}$ that converge to $u$ locally uniformly on $B_{\rho_1}$. By Korevaar's gradient estimate, up to a subsequence, we may assume that $u_k \to u$ in $C^{0,\alpha}_{\mathrm{loc}}(B_{\rho_1}) \cap W^{1,p}_{\mathrm{loc}}(B_{\rho_1})$ for all $\alpha\in (0,1)$ and $p>1$. Hence, up to a further subsequence, we may also assume that $Du_k \to Du$ a.e. in $B_{\rho_1}$. 

Fix a point $x_0\in B_{\rho_1}$ such that $u$ is differentiable at $x_0$ and $Du_k(x_0) \to Du(x_0)$. Such points exist almost everywhere in $B_{\rho_1}$. Let $g_k(x)= u_k(x) - u_k(x_0) - Du_k(x_0)\cdot (x-x_0)$ and $g(x)= u(x) - u(x_0) - Du(x_0)\cdot (x-x_0)$. Then $g_k $ converges to $ g $ locally uniformly on $B_{\rho_1}$. Since $x_0$ is a differentiable point of $u$, there exists a modulus $\sigma(r)=o(r)$ as $r\to 0$ such that $|g(x)|\leq \sigma(|x|)$.   For $k$ sufficiently large, we have $|g_k(x)|\leq 2\sigma(|x|)$. By Proposition \ref{PROP: gradient and Lip est revisited} $(ii)$, there exists a small universal constant $r_0$ independent of $k$, such that
 \begin{equation}
        \sup_{ \substack{x,y\in B_{r}(x_0) \\ x\neq y} } d_{x,y}^{n+1} \dfrac{|g_k(x)-g_k(y)|}{|x-y|} \leq C \int_{B_{r}(x_0)}|g_k|, \quad \text{for all small}\ r<r_0.
\end{equation} 
where $C$ is independent of $k$ and $r$. Passing to the limit as $k\to\infty$, we obtain 
\begin{equation}\label{eq: 7.9}
        \sup_{ \substack{x,y\in B_{r}(x_0) \\ x\neq y} } d_{x,y}^{n+1} \dfrac{|g(x)-g(y)|}{|x-y|} \leq C \int_{B_{r}(x_0)}|g|, \quad \text{for all small}\ r<r_0, 
\end{equation}
where $C$ is independent of $r$. Thus the Lipschitz estimate \eqref{eq: 7.9} holds for almost every $x_0\in B_{\rho_1}$. By replacing the ball $B_{\rho_1}$ with any ball $B_{\rho_1}(y)\subset B_1$, we conclude that the Lipschitz estimate \eqref{eq: 7.9} holds for almost every $x_0\in B_1$, which verifies hypothesis $(\mathrm{H}3)$. \hfill\qed

\section{Interior regularity for $C^1$ viscosity solutions}\label{Section: Regularity}
\noindent\emph{Proof of Theorem \ref{THM: regularity for C^1 solutions}.} By subtracting a constant from $u$, we may assume that $u(0)=0$. Since $u$ is $C^1$, we may rotate $\Sigma$ so that $Du(0)=0$.  From the $C^1$ regularity of $u$, there exists a modulus $\sigma(r)=o(r)$ as $r\to 0$ such that
\[ \sup_{B_r}|u| \leq \sigma(r) \quad \text{for all}\ r<r_0, \]
where $r_0$ depend only on $u$. By taking $r_0$ smaller if necessary, we may ensure that the curvature condition from \cite[Theorem 4.1]{Ivochkina} holds. By solving the Dirichlet problem with smooth approximating boundary data, we obtain a sequence of smooth solutions $\{u_k\}$ that converge locally uniformly to $u$ on $B_{r_0}$. Without loss of generality, we may assume that $\|u_k\|_{L^{\infty}(B_{r_0})}\leq 2\|u\|_{L^{\infty}(B_{r_0})}$.

First, applying the rescaled Korevaar's gradient estimate (Proposition \ref{PROP: Korevaar grad est}), we obtain a rough gradient bound for $u_k$:
\begin{align*}
    \sup_{B_{r_0/2}}|Du_k| &\leq C(4) \exp\left\{ C(4) (1+r_0) \left(\dfrac{\|u_k\|_{L^{\infty}(B_{r_0})}}{r_0} \right)^2 \right\} \\
    &\leq C(4) \exp \left\{ C(4) \left( \dfrac{\sigma(r_0)}{r_0} \right)^2  \right\}:=\Gamma,
\end{align*}
where $\Gamma$ is independent of $k$. Next, we apply the linearly dependent gradient estimate (Proposition \ref{PROP: linear dependent grad est}) to $u_k$. For any $r< r_0/2$, we have
\begin{align*}
    \sup_{B_{r/2}}|Du_k| \leq \dfrac{C(4,\Gamma)}{r} \sup_{B_r}|u_k|\leq \dfrac{C(4,\Gamma)}{r} \left( \sup_{B_r}|u_k-u| + \sigma(r) \right).
\end{align*}
We fix a universally small $r_1>0$ such that $\sigma(r_1)/r_1 < \mu/2$, where $\mu$ is the dimensional constant appearing in the small gradient condition of Corollary \ref{COR: rescaled doubling}. For sufficiently large $k$, we can ensure that 
\[ \sup_{B_{\frac{r_1}{2}}}|Du_k|\leq \mu. \]
Thus, by Corollary \ref{COR: rescaled doubling}, the doubling inequality holds for $u_k$ uniformly on $B_{r_1/6}$. That is, for any $y\in B_{r_1/ 18}$ and $0<\rho < r_1/9$, we have
\begin{equation}
    \sup_{B_{\frac{r_1}{6}}} H(u_k) \leq C \sup_{B_{\rho}(y)} H(u_k).
\end{equation}
We emphasize that the constant $C$ is independent of $k$.

Finally, following the same argument as in Step 3 of the proof of Theorem \ref{THM: curvature est}, using the Alexandrov regularity (Proposition \ref{PROP: Alexandrov revisited}) and Savin's small perturbation theorem \cite{Savin, Fan}, we can find a small ball $B_{\rho}(y)$ with $y\in B_{r_1/18}$ and $0<\rho< r_1/9$ , such that $\sup_{B_{\rho}(y)} H(u_k) \leq C$ uniformly in $k$. Thus, the doubling inequality yields $H(u_k)\leq C$ uniformly in $k$ on $B_{r_1/6}$. By the Evans-Krylov-Safonov theory, $u_k\to u$ locally smoothly in $B_{r_1/6}$. This implies the smoothness of $u$ near $0$. By replacing $0$ with any point in $B_1$, the interior regularity $u\in C^{\infty}_{\mathrm{loc}}(B_1)$ follows. \hfill\qed

\bibliographystyle{amsalpha}
\bibliography{ref}
\end{document}